\documentclass{amsart}

\usepackage[english]{babel}

\usepackage{amsmath,amsthm,amsfonts,amssymb,amstext,euscript,verbatim,bbm}
\usepackage{mathrsfs}
\usepackage{hyperref}
\usepackage{doi}

\usepackage{enumerate}
\usepackage{graphicx}
\usepackage{caption}
\usepackage{amscd}

\usepackage{color}

\newtheorem{thm}{Theorem}
\newtheorem{remark}[thm]{Remark}
\newtheorem{prop}[thm]{Proposition}

\newtheorem{lemma}[thm]{Lemma}

\newcommand{\nwc}{\newcommand}
\nwc{\D}{\partial}
\nwc{\grad}{\nabla}

\newcommand{\eqdef}{\overset{\mbox{\tiny{def}}}{=}}
\def\ss {\mu}
\def\rr {\rho}
\def\Ddim {d}
\newcommand{\ba}{\begin{equation}}
\newcommand{\ea}{\end{equation}}

\newcommand{\bea}{\begin{eqnarray}}
\newcommand{\eea}{\end{eqnarray}}

\begin{document}

\title[Global stability for solutions to the exponential PDE $h_t=\Delta e^{-\Delta h}$]{Global stability for solutions to the exponential PDE describing epitaxial growth}

\author{Jian-Guo Liu}
\address{Department of Physics and Department of Mathematics, Duke University, Durham, NC 27707}
\email{Jian-Guo.Liu@duke.edu}
\thanks{J.-G. L. was partially supported by the NSF grant DMS 1514826 and the NSF Research Network Grant no. RNMS11-07444 (KI-Net)}

\author{Robert M. Strain}
\address{Department of Mathematics, University of Pennsylvania, Philadelphia, PA 19104, USA.}
\email{strain@math.upenn.edu}
\thanks{R.M.S. was partially supported by the NSF grant DMS-1500916.}

%\date{May 5, 2018}
%\date{\today}
%\date{\today; %\Red{(DRAFT)}
%}

\begin{abstract}
In this paper we prove the global existence, uniqueness, optimal large time decay rates, and uniform gain of analyticity for the exponential PDE $h_t=\Delta e^{-\Delta h}$ in the whole space $\mathbb{R}^d_x$.  We assume the initial data is of medium size in the critical Wiener algebra $\Delta h \in A(\mathbb{R}^d)$.   This exponential PDE was derived in \cite{D24} and more recently in \cite{MarzuolaWeare2013}.
\end{abstract}

\setcounter{tocdepth}{1}
%\setcounter{tocdepth}{2}
%\chapter is level 0
%\section is level 1
%\subsection is level 2
%\subsubsection is level 3
%\paragraph is level 4
%\subparagraph is level 5

\maketitle
\tableofcontents
\section{Introduction and main results}

Epitaxial growth is an important physical process for forming solid films or
other nano-structures.  Indeed it is the only affordable method of high quality crystal 
growth for many semiconductor materials. It is also an  important tool to produce 
some single-layer films to perform experimental researches, highlighted by the 
recent breakthrough experiments on the quantum anomalous Hall effect 
and superconductivity above 100 K leaded by Qikun Xue \cite{Xue13, Xue15}.

This subject has been the focus of research from both physics and mathematics 
since the classic description of step dynamics in the work of Burton, Cabrera, Frank in 1951 \cite{BCF51}, 
 Weeks \cite{Weeks79} in the 1970's, the
KPZ stochastic partial differential equation description beyond roughness transition in 1986 \cite{KPZ86}, 
and the mathematical analysis of Spohn in 1993 \cite{Spohn93}. 
We refer to the books \cite{PimpinelliVillain:98, Zangwill:88} for a physical explanation of epitaxy growth. 
For more recent modeling and analysis, we refer to in particular to \cite{GigaKohn, GigaGiga, Leoni2014, Leoni2015, LuLiuMargetis:15} and the references therein. 

Epitaxy growth occurs as atoms, deposited from above, adsorb and
diffuse on a crystal surface. 
Modeling the rates that the atoms hop and break bonds leads in the continuum limit to the degenerate 4th-order PDE $h_t=\Delta e^{-\Delta_p h}$, which involves the exponential
nonlinearity and the p-Laplacian $\Delta_p$ with p=1, for example.  
In this paper, we will focus on this class of exponential PDE  for the 
the case $p=2$ and we give a short derivation of the model below.

Let $h(x,t)$ be the height of a thin film.  We consider the dynamics of atom deposition, detachment and diffusion on a crystal surface in the epitaxy growth process. 
In absence of atom deposition and in the continuum limit, the above process can be well described by 
 Fick's law: 
$$
  h_t + \nabla \cdot J = 0, \quad J = - D_s \nabla \rho_s.
$$
Here $D_s$ is the surface diffusion constant and $\rho_s$ is the equilibrium density of adatoms on a substrate of the thin film.  It is described by the grand canonical ensemble $e^{-(E_s - \mu_s)/k_BT}$ up to a normalization constant, where $E_{s}$ is the energy of pre adatom, $\mu_s$ is the chemical potential pre adatom, $k_B$ is the Boltzmann constant and $T$ is the temperature. We lump $e^{-E_s/k_BT}$ and the normalization constant into a reference density $\rho^0$
and then we arrive the Gibbs-Thomson relation $\rho_s = \rho^0 e^{\mu_s/k_BT}$ which is connected to the theory of molecular capillarity \cite{Rowlinson}. 

In the continuum limit, the chemical potential $\mu_s$ is computed by the the variation of free energy of the thin film. 
A simple broken-bond models for crystals  consists of height columns described by $h=(h_i)_{i=1,\ldots N}$ with screw-periodic boundary conditions in the form
\begin{equation*}
	h_{i+N}=h_i+\alpha a N \quad \forall i~,
\end{equation*}
where $\alpha$ is the average slope and $a$ is the side length. The  column $h_i$ is derived into $h_i/a$ square boxes where an atom is placed to the center of each box. The atoms then connect to the nearest neighbor atoms with a bond from up, down, left and right. These bonds contain almost all the energy of the system. Hence we set  the total energy of the system equal to
\begin{equation*}
E(h)=-\gamma \cdot (\mbox{\# of bonds}),
\end{equation*}
where $\gamma$ is the energy per bond.  The negative sign represents that the atoms prefer to stay together.  It requests an amount of 
$\gamma$ energy to brake the bond and separate two atoms. 
With the identity $x+|x| = 2 x_{+}$ and some elementary computations, we can decompose the total energy $E(h)$ into the bulk contribution $E_b$ and the surface contribution $E_s$.
The bulk contribution is given by
\begin{equation}\notag
E_b=- \frac{2\gamma}a \sum_{j=1}^{N} h_{i} +  \frac{\gamma \alpha } 2N.
\end{equation}
Due to the conservation of mass $\sum_{j=1}^{N} h_{i}$, we know that $E_b$ is independent of time and we can drop it from the energy computation.
The surface contribution $E_s$ is given by 
\begin{equation*}
E_s=\frac{\gamma}{2a}\sum_{i=1}^N|h_i-h_{i-1}|~.	
\end{equation*}
This free energy agrees with the computation in \cite{Weeks79}.
In general the free energy takes the form
$$
   E(h) = \frac1p \int |\nabla h|^p dx, 
$$
or some linear combinations of those \cite{Srolovitz94}. 

Now we can compute the chemical potential: $\mu_s = \frac{\delta E}{\delta h} = - \Delta_p h$
and the PDE becomes
\begin{equation} \label{eq:p}
h_t =\Delta e^{ - \Delta_p h}
\end{equation}
where, for simplicity, we have taken the constant coefficients $D_s \rho^0=1$, $k_BT=1$. 
This equation was first derived in \cite{D24} and more recently in \cite{MarzuolaWeare2013}.
A linearized Gibbs-Thomson relation $\rho_s = \rho^0 e^{\mu_s/k_BT} \approx 1 + \mu_s/k_BT$ is usually used in the physical modeling and it results the following PDE
\begin{equation}\label{eq:linear}
 h_t = \frac{D_s \rho^0}{k_BT} \Delta \Delta_p h.
\end{equation}
Giga-Kohn\cite{GigaKohn} proved that there is a finite time extinction for \eqref{eq:linear} when $p > 1$.
For the difficult case of $p=1$, Giga-Giga \cite{GigaGiga} 
developed a $H^{-1}$ total variation gradient flow to analize this equation and they showed that the solution may instantaneously develop  a jump discontinuity in the explicit example of important crystal facet dynamics. This explicit construction of the jump discontinuity solution for facet dynamics was extended to the exponential PDE \eqref{eq:p} in \cite{LLMM}.

The exponential PDE \eqref{eq:p} exhibits many distinguished behaviors in both the physical and  the mathematical senses. 
The most important one is the asymmetry in the diffusivity for the convex and concave parts of height surface profiles. This can be seen directly if we recast \eqref{eq:p} into the following Cahn-Hilliard equation with curvature-dependent mobility with \cite{LiuXu16}
\begin{equation} \notag %\label{eq:p}
h_t = \nabla \cdot \mathcal M \nabla  \frac{\delta E}{\delta h}, \quad  \mathcal M = e^{ - \Delta_p h} \,.
\end{equation}
The exponential nonlinearity drastically distinguishes the diffusivity for the convex and concave surface and leads to the singular behavior of the solution.

In \cite{LiuXu16}, a steady solution where $\Delta h$ contains a delta function was constructed and the global existence of weak solutions with $\Delta h$ as  a Radon measure was proved for the case $p=2$. A gradient
flow method in a metric space was studied together with global existence and a free energy-dissipation inequality was obtained in \cite{GaoLiuLu}.

In the present paper, we will study the case $p=2$ in the exponential PDE \eqref{eq:p}:
\begin{equation}\label{pde}
h_t=\Delta e^{-\Delta h} \quad \mbox{ in } \mathbb R^d_x \,.
\end{equation}
We will consider initial data $h_0(x)$.   We will take advantage of the Wiener Algebra $A(\mathbb{R}^d)$.  In particular in Section \ref{sec:mainResults} our main results show that if $\Delta h_0 \in A(\mathbb{R}^d)$ with explicit norm size less than $\frac{52}{500}$, assuming additional conditions, then we can prove the global existence, uniqueness, uniform gain of analyticity, and the optimal large time decay rates (in the sense of Remark \ref{remark:optimal}).  We note that the invariant scaling of \eqref{pde} is 
$h^\lambda(t,x) = \lambda^{-2} h(\lambda^4 t, \lambda x)$, and the condition $\Delta h_0 \in A(\mathbb{R}^d)$ is scale invariant (the exact space we use is $\mathcal{\dot{F}}^{2,1}$ as defined below).

\bigskip

In the next section we will introduce the necessary notation.

\subsection{Notation} 
We introduce the following useful norms:
\bea\label{weightednorm}
\|f\|_{\mathcal{\dot{F}}^{s,p}}^{p}(t) \eqdef \int_{\mathbb{R}^d} |\xi|^{sp} |\hat{f}(\xi,t)|^{p} d\xi, \quad s > -d/p,\quad 1 \le p\le 2.
\eea
Here $\hat{f}$ is the standard Fourier transform of $f$:  
\begin{equation} \label{fourierTransform}
\hat{f}(\xi) \eqdef  \mathcal{F}[f](\xi) = \frac{1}{(2\pi)^{d/2}}\int _{\mathbb{R}^{d}} f(x) e^{- i x\cdot \xi} dx.
\end{equation}
When $p=1$ we denote the norm by
\begin{equation} \label{Snorm}
\|f\|_{s} \eqdef \int_{\mathbb{R}^{d}} |\xi|^{s}|\hat{f}(\xi)| \ d\xi.
\end{equation}
We will use this norm generally for $s>-d$ and we refer to it as the \textit{s-norm}.
To further study the case $s = -d$, then for $s\ge -d$ we define the  \textit{Besov-type s-norm}:
\begin{equation} \label{normSinfty}
\|f\|_{s,\infty} \eqdef \Big\|\int_{C_{k}} |\xi|^{s}|\hat{f}(\xi)| \ d\xi\Big\|_{\ell^{\infty}_{k}}
=
\sup_{k \in \mathbb{Z}} \int_{C_{k}} |\xi|^{s}|\hat{f}(\xi)| \ d\xi,
\end{equation}
where for $k \in \mathbb{Z}$ we have
\begin{equation} \label{Cj}
C_{k} = \{\xi \in \mathbb{R}^d : 2^{k-1}\leq |\xi| < 2^{k}\} \,.
\end{equation}
Note that we have the inequality
\begin{equation}\label{ineqbd}
\|f\|_{s,\infty}  \leq \int_{\mathbb{R}^{d}} |\xi|^{s}|\hat{f}(\xi)| \ d\xi = \|f\|_{s}.	
\end{equation}
We note that $$\|f\|_{-d/p,\infty} \lesssim \|f\|_{L^p(\mathbb{R}^d)}$$ for $p\in [1,2]$ as is shown in \cite[Lemma 5]{MR3683311}.

Further, when $p=2$ we denote the norm (for $s > -d/2$) by
\bea\label{weighted2norm}
\|f\|_{\mathcal{\dot{F}}^{s,2}}^{2} \eqdef \int_{\mathbb{R}^d} |\xi|^{2s} |\hat{f}(\xi)|^{2} d\xi
=   \|f\|_{\dot{H}^s}^2 
    = \|(-\Delta)^{s/2}f\|_{L^2(\mathbb R^d)}^2.
\eea
We also introduce following norms with analytic weights:
\bea\label{analyticnorm}
\|f\|_{\mathcal{\dot{F}}^{s,p}_{\nu}}^{p}(t) \eqdef \int_{\mathbb{R}^d} |\xi|^{sp}e^{p\nu(t) |\xi|}|\hat{f}(\xi,t)|^{p} d\xi, 
\quad s\geq 0,\quad  p\in  [1,2],
\eea
for a positive function $\nu(t)$.   %This notation was used in \cite{GGJPS}.   

We also introduce the following notation for an iterated convolution
$$
f^{*2}(x) = (f * f)(x) = \int_{\mathbb{R}^d} f(x-y) f(y) dy,
$$
where $*$ denotes the standard convolution in $\mathbb{R}^d$.  Furthermore in general
$$
f^{*j}(x) = (f * \cdots * f)(x),
$$
where the above contains $j-1$ convolutions of $j$ copies of $f$.  Then by convention when $j=1$ we have 
$f^{*1} =f$, and further we use the convention $f^{*0} =1$.

We additionally use the notation $A \lesssim B$ to mean that there exists a positive inessential constant $C>0$ 
such that $A \le C B$.

\subsection{Main results}  \label{sec:mainResults}

In this section we present our main results.  Our Theorem \ref{Main.Theorem} below shows the global existence of solutions under a medium sized condition on the initial data as in Remark \ref{constant.size}.

\begin{thm}\label{Main.Theorem}
Consider initial data $h_0 \in \mathcal{\dot{F}}^{0,2}$ further satisfying $\| h_0 \|_{2} < y_*$  where $y_*>0$ is given explicitly in Remark \ref{constant.size}.  Then there exists a global in time unique solution to \eqref{pde} given by $h(t) \in C^0_t \mathcal{\dot{F}}^{0,2}$ and we have that
\begin{equation}\label{21normbound}
 \| h \|_2(t)  + \sigma_{2,1} \int_0^t   \|  h \|_6(\tau) d\tau
  \le    \| h_0 \|_2 
\end{equation}
with $\sigma_{2,1}>0$ defined by \eqref{sigma.constant.def}.
\end{thm}

In the next remark we explain the size of the constant.

\begin{remark}\label{constant.size}  We can compute precisely the size of the constant $y_*$ from Theorem \ref{Main.Theorem}.  In particular the condition that it should satisfy is that
$$
f_2(y_*)   = (y_*^3 + 6y_*^2+7y_*+1)e^{y_*}-1
=
\sum_{j=1}^{\infty}  \frac{(j+1)^3  }{j!}  y_*^{j}<1
$$
Such a $y_*$ can be taken to be $y_* \in (0,1/10]$ or $y_* \in (0, 52/500]$.  For this reason we call the initial data ``medium size''.    However $y_* \ge 105/1000$ does not work.
\end{remark}

Now in the next theorem we prove the large time decay rates, and the propagation of additional regularity, for the solutions above.

\begin{thm}\label{Cor.Main.Theorem}  We assume all the conditions in Theorem \ref{Main.Theorem}.    We also assume that $\|h_0\|_{-d,\infty} <\infty$ but not necessarily small.  

In particular for any $s> \max\{-2,-d\}$ we have that 
\begin{equation}\label{S1normbound}
 \| h \|_s(t)    \lesssim    \| h_0 \|_s, 
\end{equation}
assuming additionally that $\| h_0 \|_s <\infty$ but not necessarily small.  

Also for any $s \ge 0$ and any $p\in (1,2]$ we have that
\begin{equation}\label{S1normbound}
 \| h \|_{\mathcal{\dot{F}}^{s,p}}(t)    \lesssim    \| h_0 \|_{\mathcal{\dot{F}}^{s,p}}, 
\end{equation}
assuming that additionally $\| h_0 \|_{\mathcal{\dot{F}}^{s,p}} <\infty$ but not necessarily small.

 In particular if  $h_0 \in \mathcal{\dot{F}}^{s,1}$ and $h_0 \in \mathcal{\dot{F}}^{2,2}$ (these norms are not assumed to be small) then we conclude  the large time decay rate
 \begin{equation}\label{fastest.decay.rate}
\| h(t) \|_s \lesssim  (1+t)^{-(s+d)/4},
\end{equation}
where $d$ is the spatial dimension in \eqref{pde}.

\end{thm}

Then in the next theorem we explain the instant gain of uniform analyticity, at the optimal linear analytic radius growth rate of $t^{1/4}$, and the uniform large time decay rate of the analytic norms.

\begin{thm}\label{Analytic.Theorem}
We assume all the conditions in Theorem \ref{Main.Theorem}.  Additionally suppose that $\|h_0\|_{-d,\infty} <\infty$, $h_0 \in \mathcal{\dot{F}}^{s,1}$ for some fixed $s\ge 0$ and $h_0 \in \mathcal{\dot{F}}^{2,2}$.  

Then there exists a positive increasing function $\nu(t)>0$ such that $\nu(t) \approx t^{1/4}$ for large $t \gtrsim 1$.  For this $\nu(t)$, the solution $h(t,x)$  from Theorem \ref{Main.Theorem} further gains instant analyticity: $h(t) \in C^0_t \mathcal{\dot{F}}^{s,1}_\nu$.  And the analytic norm decays at the same rate:
 \begin{equation}\label{analytic.decay.rate}
\| h(t) \|_{\mathcal{\dot{F}}^{s,1}_\nu} \lesssim  (1+t)^{-(s+d)/4}.
\end{equation}
\end{thm}

In the remark below we further explain the optimal linear uniform time decay rates and the optimal linear gain of analyticity with radius $\nu(t) \approx t^{1/4}$.

\begin{remark}\label{remark:optimal}
Notice that the decay rates which we obtain in \eqref{fastest.decay.rate} and  \eqref{analytic.decay.rate} (and also in \eqref{decay1} below) are the same as the optimal large time decay rates for the linearization of \eqref{pde}, which is given by 
$$
h_t + \Delta^2 h = 0,
$$
obtained by removing the non-linear terms in the expansion of the nonlinearity as in \eqref{pdeSum} below.

In particular it can be shown by standard methods that if $g_0(x)$ is a tempered distribution vanishing at infinity and satisfying 
$\|g_{0}\|_{\rr,\infty} < \infty$, then   one further has
$$
\| g_0\|_{\rr,\infty}
\approx
\left\| t^{(s-\rr)/\gamma} \left\|  e^{t (-\Delta)^{\gamma/2}} g_0 \right\|_{s}  \right\|_{L^\infty_t((0, \infty) )}, \quad \text{for any $s\ge \rr$, $\gamma>0$.}
$$
This equivalence then grants the optimal time decay rate of $t^{-(s-\rr)/4}$ for $\left\|  e^{t \Delta^2} g_0 \right\|_{s}$ that is the same as the non-linear time decay rates in \eqref{fastest.decay.rate},  \eqref{analytic.decay.rate} and \eqref{decay1}.  

Further, directly using the Fourier transform methods, then the linearized equation $h_t + \Delta^2 h = 0$ also directly satisfies the time uniform time decay and gain of analyticity as in \eqref{analytic.decay.rate} with $\nu(t) \approx t^{1/4}$, and these linear rates are optimal.  
\end{remark}

When we say in this paper that the large time decay rates are optimal, we mean that we obtain the optimal linear decay rate as just described in Remark \ref{remark:optimal}.

\subsection{Methods used in the proof}  A key point in our paper is to do a Taylor expansion of the exponential non-linearity as in \eqref{pdeSum} below.  Then one can take advantage of the fact that after taking the Fourier transform, then the products in the expansion are transformed into convolutions.  Therefore one can use the structure of spaces such as $\mathcal{\dot{F}}^{2,1}$ to get useful global in time estimates like \eqref{energy3} without experiencing significant loss.  Here we mention previous work such as \cite{JEMS, CCGRPS} where a related strategy was employed for the Muskat problem.  Then we can obtain the optimal large time decay rates in the whole space using the global in time bounds that we obtain such as in \eqref{energy3} in combination with Fourier splitting techniques.  The techniques to obtain the decay rates in the whole space have a long history, and we just briefly refer to the methods in \cite{MR3683311,SohingerStrain} and the discussion therein.  To prove the uniform gain of analyticity, we perform a different splitting involving derivatives of the analytic radius of convergence from \eqref{nu.def.global}, and we acknowledge the methods from \cite{MR1608488,MR1749867} and \cite{GGJPS} that are used for different equations.  

We also mention that, after the work this paper was completed, the very recent paper \cite{GBM} was posted  showing  the global existence of at least one weak solution to the exponential PDE \eqref{pde} and the exponential large time decay, working on the torus $\mathbb{T}^d_x$.  This paper also uses the Taylor expansion of the exponential nonlinearity, and the condition $\Delta h_0 \in A(\mathbb{T}^d)$ with an equivalent size condition.

\subsection{Outline of the paper}  The rest of the paper is organized as follows.  In Section \ref{sec:aprori} we prove the a priori estimates for the exponential PDE \eqref{pde} in the spaces $\mathcal{\dot{F}}^{s,p}$ for $\in [1,2]$.    Then in Section \ref{sec:largedecay} we prove the large time decay rates in the whole space for a solution.  After that in Section \ref{uniform.bound.sec} we prove the uniform bounds in the Besov-type s-norms with negative indicies including the critical index $\| h \|_{-d,\infty}$ where $d$ is the dimension of $\mathbb{R}^d_x$.   In Section \ref{sec:uniqueness} we prove the uniqueness of solutions.   Then in Section \ref{sec:local} we sketch a proof of local existence and local gain of analyticity using an approximate regularized equation.  And in Section \ref{SizeConstant} we explain how the results from the previous sections grant directly the proofs of Theorem \ref{Main.Theorem} and Theorem \ref{Cor.Main.Theorem}.  Lastly in Section \ref{DecaySectionA} we explain how to obtain Theorem \ref{Analytic.Theorem}.  This in particular uses the previous decay results \eqref{fastest.decay.rate} as well as previous results such as \cite{MR1608488,MR1749867}.  In the Appendix \ref{sec:appendixA} we present some plots of a few numerical simulations that were carried out for the exponential PDE \eqref{pde} by Prof. Tom Witelski.

\section{A priori estimates in $\mathcal{\dot{F}}^{s,p}$}\label{sec:aprori}

In this section we prove the apriori estimates for the exponential PDE \eqref{pdeSum} in the spaces $\mathcal{\dot{F}}^{s,p}$ for $p \in [1,2]$.  The key point is that we can prove a global in time Lyapunov inequality such as \eqref{energy3} under an $O(1)$ smallness condition on the initial data.

\subsection{A priori estimate in $\mathcal{\dot{F}}^{2,1}$}  We first do the case of $\mathcal{\dot{F}}^{2,1}$ in order to explain the main idea in the simplest way.   The equation \eqref{pde} can be recast by Taylor expanion as
\begin{equation}\label{pdeSum}
   h_t + \Delta^2 h = \Delta \sum_{j=2}^{\infty}  \frac{(-\Delta h)^j}{j!}
\end{equation}
We look at this equation \eqref{pdeSum} using the Fourier transform \eqref{fourierTransform} so that equation (\ref{pde}) is expressed as
\begin{equation} \label{Fourier1}
   \partial_t  \hat h(\xi,t) +  |\xi|^4 \hat h(\xi,t) 
   = -    |\xi|^2\sum_{j=2}^{\infty}  \frac{1}{j!}  (|\cdot|^2 \hat h)^{*j} (\xi,t)  \,.
\end{equation}
We multiply the above by $|\xi|^2$ to obtain 
\begin{equation} \label{Fourier2}
   \partial_t  |\xi|^2 \hat h(\xi,t) + |\xi|^6 \hat h(\xi,t) 
   =  
   -    |\xi|^4 \sum_{j=2}^{\infty}  \frac{1}{j!}  (|\cdot|^2 \hat h)^{*j} (\xi,t) 
\end{equation}
We will estimate this equation on the Fourier side in the following.

Our first step will be to estimate the infinite sum in \eqref{Fourier2}.   To this end notice that for any real number $s \ge 0$ the following triangle inequality holds:
\begin{equation}
\label{triangleS}
    |\xi|^{s}  
    \le j^{s-1} ( |\xi-\xi_1|^{s} + \cdots + |\xi_{j-2}-\xi_{j-1}|^{s} + |\xi_{j-1}|^{s}).
\end{equation}
We have further using the inequality \eqref{triangleS} that
\begin{multline} \label{convolution1}
 \int_{\mathbb{R}^d} |\xi|^{s} | (|\cdot|^2 \hat h)^{*j} (\xi)|  \,d\xi  
 \le  j^{s}  \int_{\mathbb{R}^d} |(|\cdot|^{s+2} \hat h)*(|\cdot|^2 \hat h)^{*(j-1)} |  \,d\xi 
 \\
 \le j^{s} \| h \|_{s+2} \|  h \|_2^{j-1}. 
\end{multline}
Above we used Young's inequality repeatedly with $1+1 = 1+1$.

Using (\ref{convolution1}) for $s=4$, after integrating (\ref{Fourier2}) we obtain
\begin{equation} \label{energy1}
  \frac{d}{dt} \| h \|_2 + \|  h \|_6
  \le \| h \|_6 \sum_{j=2}^{\infty}  \frac{j^4}{j!} \|  h \|_2^{j-1}  . 
\end{equation}
Now we denote the function
\begin{equation}\label{infinite.series.fcn}
   f_2(y) = \sum_{j=2}^{\infty}  \frac{j^{4}  }{j!} y^{j-1} = \sum_{j=1}^{\infty}  \frac{(j+1)^3  }{j!}  y^{j}
\end{equation}
Then \eqref{infinite.series.fcn} defines an entire function which is strictly increasing for $y \ge 0$ with $f_2(0)=0$.  In particular we choose a value $y_*$ such that 
$f_2(y_*)=1$.

Then (\ref{energy1}) can be recast as
\begin{equation} \label{energy2}
  \frac{d}{dt} \| h \|_2 + \| h \|_6
  \le  \| h \|_6 f_2 \big( \| h \|_2 \big)
\end{equation}
If the initial data satisfies 
\begin{equation} \label{condition1}
 \| h_{0} \|_2 < y_*, 
\end{equation}
then we can show that $ \|  h (\cdot, t) \|_2$ is a decreasing function of $t$. In particular 
$$
f_2 \big( \| h(\cdot,t) \|_2 \big)
\le
f_2 \big( \| h_{0} \|_2 \big) < 1.
$$
Using this calculation then (\ref{energy2}) becomes
\begin{equation} \label{energy3}
  \frac{d}{dt} \| h \|_2 
  +
  \sigma_{2,1} \|  h \|_6
  \le   0,
\end{equation}
where
\begin{equation}\label{sigma.constant.def}
\sigma_{2,1} \eqdef 1-f_2( \|  h_{0} \|_2 ) >0.
\end{equation}
In particular if \eqref{condition1} holds, then it will continue to hold for a short time, which allows us to establish \eqref{energy3}.    The inequality (\ref{energy3}) then defines a free energy and dissipation production.

At the end of this section we look closer at the function $f_2(y)$:
$$
   f_2(y) = \sum_{j=1}^{\infty}  \frac{(j+1)^3  y^{j}}{j!} 
   = \sum_{j=1}^{\infty}  \frac{( j(j-1)(j-2) + 6j(j-1) + 7j +1)  y^{j}}{j!}, 
$$
which gives
\begin{equation}\label{fn1}
f_2(y)   = (y^3 + 6y^2+7y+1)e^y-1 \,.
\end{equation}
We know that $f_2(0)=0$ and $f_2(y)$ is strictly increasing. Let $y_*$ satisfy
\begin{equation}  \label{fn2}
 (y_{*}^3+ 6y_{*}^2+7y_*+1)e^{y_*}-1 = 1
\end{equation}
Then $f_2(y_*) = 1$ as above. 

To extend this analysis to the case where $s \ne 2$ we consider infinite series:
\begin{equation}
\label{infinite.series.fcn.s}
   f_s(y) = \sum_{j=2}^{\infty}  \frac{j^{s+2} }{j!}  y^{j-1} = \sum_{j=1}^{\infty}  \frac{(j+1)^{s+1}  }{j!} y^{j}
\end{equation}
Again $f_s(0)=0$ and $f_s(y)$ is a strictly increasing entire function for any real $s$. 
We further have a simple recursive relation 
$$
   f_s(y) = \frac{d}{dy} \big(y f_{s-1}(y) \big) , \quad 
   f_{-1}(y) =  e^y - 1.
$$
This allows us to compute $f_s(y)$ for any $s$ a non-negative integer as in \eqref{fn1}.

\subsection{A priori estimate in the high order s-norm} 
In this section we prove a high order estimate for any real number $s> \max\{-2, -d\}$:
 \begin{equation} \label{Fourier3}
   \partial_t  |\xi|^{s} \hat h(\xi,t) + |\xi|^{s+4} \hat h(\xi,t) 
   = -  |\xi|^{s+2} \sum_{j=2}^{\infty}  \frac{1}{j!} (|\xi|^2 \hat h)^{*j} (\xi,t) 
\end{equation}
Using (\ref{convolution1}) and (\ref{Fourier3}), one has
\begin{equation} \label{energy4}
  \frac{d}{dt} \| h \|_s + \|  h \|_{s+4}
  \le  \|  h \|_{s+4} \sum_{j=2}^{\infty}  \frac{j^{s+2}}{j!} \| h \|_2^{j-1} 
\end{equation}
Now we recast (\ref{energy4}) as
\begin{equation} \label{energy5}
  \frac{d}{dt} \|  h \|_s + \| h \|_{s+4}
  \le  \|  h \|_{s+4} f_s \big( \| h \|_2\big)
\end{equation}
Let $y_{s*}$ satisfy $f_s(y_{s*}) = 1$.
If 
\begin{equation} \label{condition2}
 \| h_{0} \|_2 < \min(y_{s*}, y_*), 
\end{equation}
then by \eqref{energy3} we have
$$
f_s \big( \| h(\cdot,t) \|_2 \big)
\le
f_s \big( \| h_{0} \|_2 \big) < 1 \,.
$$
Hence we conclude the energy-dissipation relation 
\begin{equation} \label{energy-diss2}
  \frac{d}{dt} \| h(\cdot,t) \|_s 
   + 
 \sigma_{s,1} \| h(\cdot,t) \|_{s+4} \,
  \le  0 , \,.
\end{equation}
when (\ref{condition2}) holds.  Here we define $\sigma_{s,1} \eqdef \big( 1-f_s( \|  h_{0} \|_2 \big)>0$.

 \subsection{A priori estimate in $\mathcal{\dot{F}}^{s,p}$.} 
 
 In this section we prove a general $\mathcal{\dot{F}}^{s,p}$ estimate for $s\ge 0$ and $1< p \le 2$.   We multiply the equation \eqref{Fourier1} by $p|\xi|^{sp}\overline{\hat h} |{\hat h}|^{p-2}(\xi,t)$, and integrate to obtain
\begin{multline} \label{Fourier6}
   \partial_t \left( |\xi|^{sp} |\hat h|^p(\xi,t) \right)
   + p|\xi|^{sp+4} |\hat h|^p(\xi,t) 
   \\
   = -  p  \sum_{j=2}^{\infty}   \frac{1}{j!}  |\xi|^{sp+2} \overline{\hat h}|{\hat h}|^{p-2}(\xi,t)(|\xi|^2 \hat h)^{*j} (\xi,t). 
\end{multline}
We will estimate this equation when $p\in (1,2]$.   To this end, we split $ps=(p-1)s +s$, and we split $2 = \frac{(p-1)}{p} 4 +\left( \frac{4}{p} - 2\right)$.   We will do a H{\"o}lder inequality with $\frac{p-1}{p}+\frac{1}{p} = 1$, then use \eqref{triangleS}, and then Young's inequality repeatedly with $1+\frac{1}{p} = \frac{1}{p}+1$ to obtain
\begin{align} 
 \int_{\mathbb{R}^d} |\xi|^{ps+2} 
 & | \overline{\hat h}(\xi) |\hat h|^{p-2}(\xi) (|\cdot|^2 \hat h(\cdot))^{*j}(\xi)  |  \,d\xi  
  \nonumber \\
 & \le 
 \| |\xi|^{s+4/p} \hat h \|_{L^p}^{p-1}  \| |\xi|^{s+\frac{4}{p} - 2} (|\xi|^2 \hat h)^{*j} \|_{L^p} 
  \nonumber \\
  &\le 
 \| |\xi|^{s+4/p} \hat h \|_{L^p}^{p-1}  j^{s+\frac{4}{p} - 2} \| (|\xi|^{s+4/p} \hat h)*(|\xi|^2 \hat h)^{*(j-1)} \|_{L^p} 
  \nonumber \\
 & \le 
j^{s+\frac{4}{p} - 2} \|h\|_{\mathcal{\dot{F}}^{s+4/p,p}}^{p}  \| h \|_2^{j-1}.  
\label{convolutionP}
\end{align} 
We now use \eqref{Fourier6} and (\ref{convolutionP})   to obtain
\begin{equation} \label{energy10}
\frac{d}{dt} \| h\|_{\mathcal{\dot{F}}^{s,p}}^p + p \| h\|_{\mathcal{\dot{F}}^{s+4/p,p}}^p
\le  p   \| h \|_{\mathcal{\dot{F}}^{s+4/p,p}}^p \sum_{j=2}^{\infty}   \frac{j^{s+4/p-2}}{j!}  \| h \|_2^{j-1}
\end{equation}
The sum in the upper bound is $f_{s+4/p-4}(\| h \|_2)$ from \eqref{infinite.series.fcn.s}.  Similar to the previous discussions, we choose the positive real number $y_{sp*}$ to satisfy $f_{s+4/p-4}(y_{sp*}) = 1$.

Then if
\begin{equation} \label{conditionF}
 \| h_{0} \|_2 < \min(y_{sp*}, y_*), 
\end{equation}
it further holds that
$
f_{s+4/p-4} \big( \| h(\cdot,t) \|_2 \big)
\le
f_{s+4/p-4} \big( \| h_{0} \|_2 \big) < 1.
$
Hence we again have the energy-dissipation relation 
\begin{equation} \label{energy-dissF}
  \frac{d}{dt} \| h(\cdot,t) \|_{\mathcal{\dot{F}}^{s,p}}^p 
   + p    \sigma_{s,p} \| h(\cdot, t) \|_{\mathcal{\dot{F}}^{s+4/p,p}}^p 
  \le  0 \,.
\end{equation}
when (\ref{conditionF}) holds.  Here $\sigma_{s,p} \eqdef \big( 1-f_{s+4/p-4} ( \| h_{0} \|_2 )  \big)>0$.

 \section{Large time decay in $\mathcal{\dot{F}}^{s,1}$}\label{sec:largedecay}

In this section we prove the following large time decay rates in the whole space

\begin{prop}\label{prop.large.decay}
Given the solution to \eqref{pde} from Theorem \ref{Main.Theorem}.  Suppose additionally that $\| h_0 \|_{s} <\infty$ for some $s> \max\{-2,-d\}$  Further suppose $\| h_0 \|_{\rr,\infty} <\infty$ for some $-d \le \rr <s$.  Assume that $\| h_0 \|^{2}_{\mathcal{\dot{F}}^{\rr+d+2,2} }$ and $\| h_0 \|^{2}_{\mathcal{\dot{F}}^{0,2} }$ are both initially finite. Then we have the following uniform decay estimate for $t\ge 0$:
\begin{equation}\label{decay1}
\| h \|_s \lesssim  (1+t)^{-(s-\rr)/4}.
\end{equation}
The implicit constant in the inequality above depends on $\| h_0 \|_2$, $\| h_0 \|_s$, $\| h_0 \|_{\rr,\infty}$, $\| h_0 \|^{2}_{\mathcal{\dot{F}}^{\rr+d+2,2} }$, and $\| h_0 \|^{2}_{\mathcal{\dot{F}}^{0,2} }$. 
\end{prop}

Notice that this decay only depends on the smallness of the $\| h_0 \|_2$ norm.  No other norm is required to be small.  Further notice that Proposition \ref{prop.large.decay} directly implies \eqref{fastest.decay.rate} in Theorem \ref{Cor.Main.Theorem}

A key step in proving \eqref{decay1} is to prove the following uniform estimate:

\begin{prop}\label{prop.bd.uniform}  Given the solution from Theorem \ref{Main.Theorem}.  Suppose additionally that $\| h_0 \|_{\rr,\infty} <\infty$ for $\rr \ge -d$.  Further assume that $\| h_0 \|^{2}_{\mathcal{\dot{F}}^{\rr+d+2,2} }$ and $\| h_0 \|^{2}_{\mathcal{\dot{F}}^{0,2} }$ are both initially finite.  Then we have
\begin{equation}\label{uniform1}
\| h \|_{\rr,\infty} \lesssim 1, \quad \rr \ge -d.
\end{equation}
\end{prop}

The proof of Proposition \ref{prop.bd.uniform}  will be addressed in Section \ref{uniform.bound.sec}.  The goal of this section is to establish \eqref{decay1} by assuming \eqref{uniform1}.

We will use the following decay lemma from Patel-Strain \cite{MR3683311}:

\begin{lemma}\label{decaylemma}
Suppose $g=g(t,x)$ is a smooth function with $g(0,x) = g_0(x)$ and assume that for some $\ss \in \mathbb{R}$, $\|g_0\|_{\ss} < \infty$ and 
$$
\|g(t)\|_{\rr,\infty} \leq C_{0}
$$ 
for some $\rr \ge -\Ddim$ satisfying  $ \rr < \ss$.
Let the following differential inequality hold for $\gamma>0$ and for some $C>0$:
	$$
	\frac{d}{dt}\|g\|_{\ss} \leq -C \|g\|_{\ss +\gamma}.
	$$ 
	  Then we have the uniform in time estimate
	$$ 
	\|g\|_{\ss}(t) \lesssim \left( \|g_0\|_{\ss}+ C_{0}\right) (1+t)^{-(\ss-\rr)/\gamma}.
	$$
\end{lemma}

\begin{remark}\label{remarkdecay}
Note that by \eqref{ineqbd} we have $\|f\|_{\rr,\infty}(t) \leq \|f\|_{\rr}(t)$. Therefore we can use Lemma \ref{decaylemma} if we can bound $\|f\|_{\rr}(t)$ for $\rr > -\Ddim$ uniformly in time.
\end{remark}

This lemma is stated in the paper \cite{MR3683311} with $\gamma=1$, however the similar proof below assumes only that $\gamma>0$.  We include the proof for completeness.

\begin{proof}
For some $\delta, \kappa>0$ to be chosen, we initially observe that
\begin{align*}
	\|g\|_{\kappa} &=\int_{\mathbb{R}^{d}}|\xi|^{\kappa}|\hat{g}(\xi)|d\xi \\
	&\geq \int_{|\xi|>(1+\delta t)^{s}}|\xi|^{\kappa}|\hat{g}(\xi)|d\xi \\
	&\geq (1+\delta t)^{s\beta}\int_{|\xi|>(1+\delta t)^{s}}|\xi|^{\kappa-\beta}|\hat{g}(\xi)|d\xi\\
	&= (1+\delta t)^{s\beta}\Big(\|g\|_{\kappa-\beta} \  - \int_{|\xi|\leq(1+\delta t)^{s}}|\xi|^{\kappa-\beta}|\hat{g}(\xi)|d\xi\Big).
\end{align*}
Using this inequality with $\kappa = \ss + \gamma$ and $\beta = \gamma$, we obtain that
\begin{multline*}
\frac{d}{dt}\|g\|_{\ss} + C(1+\delta t)^{s\gamma}\|g\|_{\ss} \leq -C\|g\|_{\ss+\gamma} + C(1+\delta t)^{s\gamma}\|g\|_{\ss} 
\\
\leq C(1+\delta t)^{s\gamma}\int_{|\xi| \leq(1+\delta t)^{s}}|\xi|^{\ss}|\hat{g}(\xi)|d\xi.
\end{multline*}
	Then, using the sets $C_k$ as in \eqref{normSinfty} and defining $\chi_{S}$ to be the characteristic function on a set $S$, the upper bound in the last inequality can be bounded as follows
\begin{align*}
	\int_{|\xi|\leq(1+\delta t)^{s}}|\xi|^{\ss}|\hat{g}(\xi)|d\xi
	&= \sum_{k\in\mathbb{Z}} \int_{C_{k}}\chi_{\{|\xi|\leq (1+\delta t)^{s}\}}|\xi|^{\ss}|\hat{g}| \ d\xi\\ 
	&\approx \sum_{2^{k}\leq (1+\delta t)^{s}} \int_{C_{k}} |\xi|^{\ss}|\hat{g}| \ d\xi\\
	&\lesssim \|g\|_{\rr,\infty}\sum_{2^{k}\leq (1+\delta t)^{s}} 2^{k(\ss-\rr)}\\
	&\lesssim \|g\|_{\rr,\infty} (1+\delta t)^{s(\ss-\rr)}\sum_{2^{k}(1+\delta t)^{-s}\leq 1} 2^{k(\ss-\rr)}(1+\delta t)^{-s(\ss-\rr)} \\
	&\lesssim \|g\|_{\rr,\infty}(1+\delta t)^{s(\ss-\rr)} \\
	&\lesssim C_{0}(1+\delta t)^{s(\ss-\rr)},
\end{align*}
where the implicit constant in the inequalities does not depend on $t$.  In particular we have used that the following uniform in time estimate holds
	$$
	\sum_{2^{k}(1+\delta t)^{-s}\leq 1} 2^{k(\ss-\rr)}(1+\delta t)^{-s(\ss-\rr)} \lesssim 1.
	$$
	%By the hypothesis of the lemma, we can also bound $\|g\|_{\rr,\infty} \leq C_{0}$. 
	Combining the above inequalities, we obtain that
	\bea\label{use}
	\frac{d}{dt}\|g\|_{\ss} + C(1+\delta t)^{s\gamma}\|g\|_{\ss} \lesssim C_{0} (1+\delta t)^{s\gamma}(1+\delta t)^{s(\ss-\rr)}.
	\eea
In the following estimate will use \eqref{use} with $s=-1/\gamma$, we suppose $a > (\ss-\rr)/\gamma >0$, and we choose $\delta > 0$ such that $a\delta = C$.  We then obtain that
\begin{align*}
	\frac{d}{dt}((1+\delta t)^{a}\|g\|_{\ss}) &=  (1+\delta t)^{a}\frac{d}{dt}\|g\|_{\ss} + a\delta \|g\|_{\ss}(1+\delta t)^{a-1} \\
	&=  (1+\delta t)^{a}\frac{d}{dt}\|g\|_{\ss} + C \|g\|_{\ss}(1+\delta t)^{a-1} \\
	&\leq (1+\delta t)^{a}\Big(\frac{d}{dt}\|g\|_{\ss} + C(1+\delta t)^{-1}\|g\|_{\ss}\Big) \\
	&\lesssim  C_{0} (1+\delta t)^{a-1-(\ss-\rr)/\gamma}.
\end{align*}
Since $a > \ss-\rr$, we integrate in time to obtain that
$$ 
	(1+\delta t)^{a}\|g(t)\|_{\ss} \lesssim \|g_0\|_{\ss} + \frac{ C_{0}}{\delta} (1+\delta t)^{a-(\ss-\rr)/\gamma}.
$$ 
We conclude our proof by dividing both sides of the inequality by $(1+\delta t)^{a}$.
\end{proof}

We have established the differential energy inequality \eqref{energy3} for the equation \eqref{pde}.  Thus to prove the time decay in \eqref{decay1}
it remains only to establish \eqref{uniform1}.

 \section{Proof of the uniform bound $\| h \|_{\rr,\infty} \lesssim 1$}\label{uniform.bound.sec}
 
 In this section we will prove the uniform bound in \eqref{uniform1}.
We recall the equation \eqref{Fourier1} again.  For $\max\{ -2, -d\} < s \le 2$ from \eqref{condition1}  and \eqref{energy3} and \eqref{energy-diss2} we have
\begin{equation}\label{uniformSbound1}
\| h \|_s(t) \lesssim \| h_0 \|_s.
\end{equation}
This will hold assuming that $\| h_0 \|_s<\infty$ and \eqref{condition1}  holds.

We will now prove the uniform bound \eqref{uniform1} when $\rr \ge -d$.   Notice that \eqref{uniformSbound1} and \eqref{ineqbd} imply \eqref{uniform1} when  $\max\{ -2, -d\} < s \le 2$.   The proof of the bound below can be used when 
$-d \le \rr \le \max\{-2,-d\}.$  

\begin{proof}[Proof of Proposition \ref{prop.bd.uniform}]
We recall \eqref{Fourier1}, and we uniformly bound the integral over $C_{k}$ for each $j\in\mathbb{Z}$ as in \eqref{normSinfty}. 
We obtain the following differential inequality 
\begin{multline}\label{energyCj}
\frac{d}{dt} \int_{C_{k}}|\xi|^{\rr}|\hat{h}(\xi,t)|d\xi + \int_{C_{k}} d\xi \ |\xi|^{\rr+4}|\hat{h}(\xi,t)| 
\\
\leq \sum_{j=2}^{\infty}  \frac{1}{j!}  \int_{C_{k}} d\xi \  |\xi|^{\rr+2} (|\xi|^2 \hat h)^{*j} (\xi,t).
\end{multline}
We will estimate the upper bound.  We can estimate the integral as
\begin{align} 
\int_{C_{k}} d\xi \  |\xi|^{\rr+2} (|\xi|^2 \hat h)^{*j} (\xi,t)
&\lesssim
2^{-d(k-1)} \int_{C_{k}} d\xi \  |\xi|^{\rr+d+2} (|\xi|^2 \hat h)^{*j} (\xi,t)
\nonumber \\
&\lesssim
  \| |\cdot|^{\rr+d+2} (|\cdot|^2 \hat h)^{*j} (\xi,t)\|_{L^\infty_\xi } 
\label{Besov1}
\end{align} 
The last inequality holds because the integral over $C_{k}$ is of size $2^{dk}$.

Notice further that $\rr +d \ge 0$ and then we can use \eqref{triangleS}.
We also use Young's inequality, first with $1 + \frac{1}{\infty} = \frac{1}{2} + \frac{1}{2}$, and again with $1+\frac{1}{2}=1+\frac{1}{2}$ repeatedly to obtain:
\begin{align} 
  \| |\cdot|^{\rr+d+2} (|\cdot|^2 \hat h)^{*j} (\cdot)\|_{L^\infty_\xi } 
  &\lesssim j^{\rr+d+2}
 \| |\cdot |^{\rr+d+4} \hat h (\cdot)\|_{L^2_\xi }
 \|(|\cdot|^2 \hat h)^{*(j-1)} (\cdot)\|_{L^2_\xi } 
\nonumber \\
%&\lesssim j^{\rr+d+2}
% \| h \|_{\mathcal{\dot{F}}^{\rr+d+4,2} }
% \|(|\cdot|^2 \hat h)^{*(k-1)} (\cdot)\|_{L^2_\xi } 
%\nonumber \\
&\lesssim j^{\rr+d+2}
 \| h \|_{\mathcal{\dot{F}}^{\rr+d+4,2} }
 \| |\cdot |^{2} \hat h (\cdot)\|_{L^2_\xi } 
  \| |\cdot |^2 \hat h (\cdot)\|_{L^1_\xi }^{j-2} 
\nonumber \\
& \lesssim j^{\rr+d+2}
 \| h \|_{\mathcal{\dot{F}}^{\rr+d+4,2} }
 \| h \|_{\mathcal{\dot{F}}^{2,2} } 
  \|  h \|_{2 }^{j-2}. 
\label{Besov2}
\end{align} 
Now we plug \eqref{Besov1} into \eqref{energyCj} to obtain
\begin{multline}\label{energyCjA}
\frac{d}{dt} \int_{C_{k}}|\xi|^{\rr}|\hat{h}(\xi,t)|d\xi + \int_{C_{k}} d\xi \ |\xi|^{\rr+4}|\hat{h}(\xi,t)| 
\\
\leq \sum_{j=2}^{\infty}  \frac{1}{j!}   \| |\cdot|^{\rr+d+2} (|\cdot|^2 \hat h)^{*j} (\cdot)\|_{L^\infty_\xi }. 
\end{multline}
We further estimate the upper bound using \eqref{Besov2} to obtain
\begin{multline*}
\sum_{j=2}^{\infty}  \frac{1}{j!}   \| |\cdot|^{\rr+d+2} (|\cdot|^2 \hat h)^{*j} (\cdot,t)\|_{L^\infty_\xi } 
\\
\lesssim
 \| h (t)\|_{\mathcal{\dot{F}}^{\rr+d+4,2} }
 \| h (t)\|_{\mathcal{\dot{F}}^{2,2} } 
 \sum_{j=2}^{\infty}  \frac{j^{\rr+d+2}}{j!}    \|  h (t)\|_{2 }^{j-2} 
  \\
\lesssim   \| h (t)\|_{\mathcal{\dot{F}}^{\rr+d+4,2} }
 \| h (t)\|_{\mathcal{\dot{F}}^{2,2} }. 
\end{multline*}
In the above we have used that 
\begin{equation}\notag %\label{boundNeeded1}
 \sum_{j=2}^{\infty}  \frac{j^{\rr+d+2}}{j!}  \| h (t)\|_{2 }^{j-2}  \lesssim 1.
\end{equation}
The above holds because the sum initially
$
 \sum_{j=2}^{\infty}  \frac{j^{\rr+d+2}}{j!}  \| h_{0}\|_{2 }^{j-2} \lesssim 1
$
converges generally.    Then we further use the estimate \eqref{energy3} to see that $\| h (t)\|_{2 } \le \| h_{0}\|_{2 }$.

We conclude from integrating \eqref{energyCjA} and using the above estimates that 
\begin{multline}\label{energyBCj}
\int_{C_{k}}|\xi|^{\rr}|\hat{h}(\xi,t)|d\xi + \int_0^t ds \int_{C_{k}} d\xi \ |\xi|^{\rr+4}|\hat{h}(\xi,s)| 
\\
\leq \int_{C_{k}}|\xi|^{\rr}|\hat{h}_0 (\xi)|d\xi+ \int_0^t ds \| h (t)\|_{\mathcal{\dot{F}}^{\rr+d+4,2} }
 \| h (t)\|_{\mathcal{\dot{F}}^{2,2} }. 
\end{multline}
Thus as long as we make the proper assumptions to bound 
\begin{multline}\label{boundNeeded2}
\int_0^t ds \| h (t)\|_{\mathcal{\dot{F}}^{\rr+d+4,2} }
 \| h (t)\|_{\mathcal{\dot{F}}^{2,2} } 
 \\
\le  
\sqrt{\int_0^t ds \| h (t) \|^{2}_{\mathcal{\dot{F}}^{\rr+d+4,2} } \,dt}
 \sqrt{\int_0^t 
 \| h (t)\|^{2}_{\mathcal{\dot{F}}^{2,2} }  \,dt}
  \lesssim 1,
\end{multline}
then the bound \eqref{energyBCj} with  \eqref{boundNeeded2}  implies Proposition \ref{prop.bd.uniform}.  

However when $\| h_0 \|_{\mathcal{\dot{F}}^{\rr+d+2,2} }$ and $\| h_0 \|_{\mathcal{\dot{F}}^{0,2} }$ are both initially finite then \eqref{energy-dissF} implies that 
\eqref{boundNeeded2} holds.   Here we have used \eqref{energy-dissF}  with $p=2$, $s=\rr+d+2$ and $s=0$, respectively.
Then we obtain the bound \eqref{uniform1} for $\rr \ge -d$.
\end{proof}

 \section{Uniqueness}\label{sec:uniqueness}

In this section we prove the uniqueness of solutions to  \eqref{pde} which satisfy \eqref{condition1}.

 \begin{prop}\label{prop.uniqueness}  Given two solutions $h_1$ and $h_2$ to \eqref{pde} with the same initial data $h_0$ satisfying \eqref{condition1}.   Then $\| h_1-h_2\|_2 =0$.  If we further assume that the initial data satisfies $ \| h_0\|_0 <\infty$, then  $\| h_1-h_2\|_0 =0$.    In particular $\| h_1- h_2\|_{L^\infty_x} = 0$.
 \end{prop}

\begin{proof}[Proof of Proposition \ref{prop.uniqueness}]  We consider the equation \eqref{pdeSum} satisfied by both  $h_1$ and $h_2$.  Then we have that   
\begin{equation}\notag %\label{pdeSum.diff}
   (h_1-h_2)_t + \Delta^2 (h_1-h_2) = \Delta \sum_{j=2}^{\infty}  \frac{(-\Delta h_1)^j-(-\Delta h_2)^j}{j!}.
\end{equation}
We further have the algebraic identity
\begin{equation}\notag
(-\Delta h_1)^j-(-\Delta h_2)^j
=
- \left(\Delta h_1 - \Delta h_2\right) \left( \sum_{m=0}^{j-1} (-\Delta h_1)^{j-1-m} (-\Delta h_2)^{m}   \right).
\end{equation}
We take the Fourier fransform to obtain
\begin{multline} \label{Fourier.diff}
   \partial_t  \left( \hat h_1(\xi,t) - \hat h_2(\xi,t) \right)+  |\xi|^4 \left( \hat h_1(\xi,t) - \hat h_2(\xi,t) \right)
   \\
   = -    |\xi|^2\left(|\cdot|^2\left( \hat h_1 - \hat h_2 \right) \right)*\sum_{j=2}^{\infty}  \frac{1}{j!} 
   \left( \sum_{m=0}^{j-1} (|\cdot|^2 \hat h_1)^{*(j-1-m)}  * (|\cdot|^2 \hat h_2)^{*m}    \right).
\end{multline}
Then we obtain that 
\begin{equation} \label{diff.ineq.s}
   \frac{d}{dt}  \| h_1 -  h_2 \|_s +  \| h_1 -  h_2 \|_{s+4}
   = \mathcal{U}_s.
\end{equation}
Above $\mathcal{U}_s$ is the integral of the right side of \eqref{Fourier.diff} multiplied by $|\xi|^s$:
\begin{multline} \notag
 \mathcal{U}_s \eqdef 
  \\
 -\int_{\mathbb{R}^d} ~d\xi ~
 |\xi|^{2+s}
 \sum_{j=2}^{\infty}  \frac{1}{j!} 
   \left( \sum_{m=0}^{j-1} \left(|\cdot|^2\left( \hat h_1 - \hat h_2 \right) \right)*(|\cdot|^2 \hat h_1)^{*(j-1-m)}  * (|\cdot|^2 \hat h_2)^{*m}    \right).
\end{multline}
We will consider the cases $s=2$ and then $s=0$.  

When $s=2$ above, we use \eqref{triangleS} and Young's inequality to obtain 
\begin{multline} \notag
\mathcal{U}_2 
   \le 
   \| h_1 -  h_2 \|_6
       \sum_{j=2}^{\infty}  \frac{j^4}{j!} 
    \max\{ \| h_1 \|_2,  \| h_2 \|_2\}^{j-1}    
   \\
      +\| h_1 -  h_2 \|_2 \max\{ \| h_1 \|_6,  \| h_2 \|_6\}
       \sum_{j=2}^{\infty}  \frac{j^5}{j!}  \max\{ \| h_1 \|_2,  \| h_2 \|_2\}^{j-2}.   
\end{multline}
Above we use $\max\{ \| h_1 \|_2,  \| h_2 \|_2\}$ only to simplify the notation.  

Now since the initial data satisfies \eqref{condition1} then we have
$$
\sum_{j=2}^{\infty}  \frac{j^4}{j!} 
    \max\{ \| h_1 \|_2,  \| h_2 \|_2\}^{j-1}    
    \le 
    \sum_{j=2}^{\infty}  \frac{j^4}{j!} 
     \| h_0 \|_2^{j-1}    <1.
$$
Also 
$\sum_{j=2}^{\infty}  \frac{j^5}{j!}  \max\{ \| h_1 \|_2,  \| h_2 \|_2\}^{j-2} \le \sum_{j=2}^{\infty}  \frac{j^5}{j!}  \| h_0 \|_2^{j-2} \lesssim 1.$
Then after integrating  \eqref{diff.ineq.s} with $s=2$ in time, we obtain that  
\begin{multline} \notag
\| h_1 -  h_2 \|_2(t) + \sigma \int_0^t ~ \| h_1 -  h_2 \|_{6}(s) ~ds
\\
\lesssim
 \int_0^t ~ \| h_1 -  h_2 \|_{2}(s)  \max\{ \| h_1 \|_6,  \| h_2 \|_6\}(s) ~ ds.
\end{multline}
Here we use $\sigma =  \sigma_{2,1}>0$  from \eqref{sigma.constant.def}.
Notice that $ \int_0^t ~  \max\{ \| h_1 \|_6,  \| h_2 \|_6\}(s) ~ ds<\infty$ by \eqref{energy3}.  
Now the Gronwall inequality implies that $\| h_1 -  h_2 \|_2=0$.

We turn to the case $s=0$ in \eqref{diff.ineq.s}.  We will obtain an upper bound for $\mathcal{U}_0$.  Similarly we will use \eqref{triangleS} with $s=2$.  Then with Young's inequality we obtain 
\begin{multline} \notag
\mathcal{U}_0 
   \le 
   \| h_1 -  h_2 \|_4
       \sum_{j=2}^{\infty}  \frac{j^2}{j!} 
    \max\{ \| h_1 \|_2,  \| h_2 \|_2\}^{j-1}    
   \\
      +\| h_1 -  h_2 \|_2 \max\{ \| h_1 \|_4,  \| h_2 \|_4\}
       \sum_{j=2}^{\infty}  \frac{j^3}{j!}  \max\{ \| h_1 \|_2,  \| h_2 \|_2\}^{j-2}.   
\end{multline}
However we know from the previous case that $\| h_1 -  h_2 \|_2=0$ therefore the second term above is zero.
Then similar to the previous case, for a $\delta>0$ we obtain
$$
\| h_1 -  h_2 \|_0(t) + \delta \int_0^t ~ \| h_1 -  h_2 \|_{4}(s) ~ds \le 0.
$$
We conclude that $\| h_1 -  h_2 \|_0=0$.
\end{proof}

We remark that the same methods can be used to prove that $\| h_1 -  h_2 \|_{\mathcal{\dot{F}}^{0,2}}=0$.

 \section{Local existence and approximation}\label{sec:local}
 
In this section we prove the local existence theorem using a suitable approximation scheme.  Since the methods in this section are rather standard, therefore we provide a sketch of the key ideas.

 \begin{prop}\label{prop.local}
Consider initial data $h_0 \in \mathcal{\dot{F}}^{0,2}$ further satisfying $\| h_0 \|_{2} < y_*$  where $y_*>0$ is given explicitly in Remark \ref{constant.size}.

Then there exists an interval $[0,T]$ upon which we have a local in time unique solution to \eqref{pde} given by 
$h(t) \in C^0([0,T]; \mathcal{\dot{F}}^{0,2})$.  This solution also gains instant analyticity as
$$
  \| h\|_{\mathcal{\dot{F}}^{2,1}_\nu} (t) \le  \| h_{0}\|_2 e^{bT},
$$
where $\nu(t) = b t$ for some fixed $b\in (0,1)$ with $b=b(\| h_0 \|_{2})$.
 \end{prop}
 
To prove this we perform a regularization of \eqref{pdeSum} as follows.  Let $\zeta_t$ be the heat kernel in $\mathbb{R}^d$ for $t>0$.  We will consider 
$\zeta_\epsilon$ with $\epsilon>0$ so that $\zeta_\epsilon$ is an approximation to the identity as $\epsilon \to 0$.  We define the regularized equation as:
\begin{equation}\label{pdeSumApprox}
   \partial_t h^\epsilon + \Delta^2 (\zeta_\epsilon * \zeta_\epsilon*  h^\epsilon)  = \Delta \sum_{j=2}^{\infty}  \frac{(-\Delta (\zeta_\epsilon * \zeta_\epsilon*  h^\epsilon))^j}{j!},
   \quad h^\epsilon_0 =  \zeta_\epsilon*  h_0.
\end{equation}
This regularized system \eqref{pdeSumApprox} can be directly estimated using all of the apriori estimates from the previous sections.   In particular all the previous estimates for \eqref{pde} in this paper continue to straightforwardly apply to the approximate problem \eqref{pdeSumApprox}.

These estimates allow us to prove a local existence theorem for the regularized system \eqref{pdeSumApprox} using the Picard theorem on a Banach space $C^0([0,T_\epsilon];\mathcal{\dot{F}}^{K,2})$ for some $K\ge 4$.  We find the abstract evolution system given by $\partial_th^\epsilon=F(h^\epsilon)$ where $F$ is Lipschitz on the open set $\{f(x)\in \mathcal{\dot{F}}^{K,2}:\|f\|_{\dot{\mathcal{F}}^{2,1}}<y_*\}$. 
Observe that $h_0^\epsilon \in \mathcal{\dot{F}}^{K,2}$ since  $h_0\in \mathcal{\dot{F}}^{0,2}$.   Further, since the convolutions are taken with the heat kernel, we can prove analyticity for $h^\epsilon$.

In particular directly following \eqref{energy-diss2} we obtain for some $\delta_1>0$ that
\begin{equation} \label{energy-diss2.approx}
  \frac{d}{dt} \| h^\epsilon(t) \|_2 
   + 
\delta_1 \| \zeta_\epsilon * \zeta_\epsilon*  h^\epsilon \|_{6} \,
  \le  0.
\end{equation}
Similarly following  following \eqref{energy-dissF} we obtain for some $\delta_2>0$ that
\begin{equation} \label{energy-dissF.approx}
  \frac{d}{dt} \| h^\epsilon(t) \|_{\mathcal{\dot{F}}^{0,2}}^2 
   + \delta_2 \| \zeta_\epsilon * \zeta_\epsilon*  h^\epsilon\|_{\mathcal{\dot{F}}^{0,2}}^2 
  \le  0 \,.
\end{equation}
From these estimates, \eqref{energy-diss2.approx} and \eqref{energy-dissF.approx}, we obtain the strong convergence needed to take a limit as $\epsilon \to 0$ in \eqref{pdeSumApprox} and obtain the unique solution from Proposition \ref{prop.local} on the uniform time interval $[0,T]$ for some $T>0$.

In the following we show how to to reach analyticity in short time as in Proposition \ref{prop.local}.  The approximation scheme in \eqref{pdeSumApprox} is well designed to reach the analytic regime in short time, and maintain the analyticity in the limit as $\epsilon \to 0$.  Below, we explain the gain of analyticity with the a priori estimate.

\subsection{Reach analyticity in a short time}
We use $\nu(t) = b t$ (for some $b>0$ to be determined) in the analytic space \eqref{analyticnorm} with $s=2$ and $p=1$.  Note that $|\xi|^3 \le |\xi|^6 + |\xi|^2$.  From \eqref{Fourier2} we obtain the following differential inequality:
\begin{equation} \label{energy13}
\frac{d}{dt} \| h\|_{\mathcal{\dot{F}}^{2,1}_\nu} 
+  (1-\nu'(t)) \| h\|_{\mathcal{\dot{F}}^{6,1}_\nu}
\le  \nu'(t) \| h\|_{\mathcal{\dot{F}}^{2,1}_\nu} +  \| h \|_{\mathcal{\dot{F}}^{6,1}_{\nu}} \sum_{j=2}^{\infty}   \frac{j^{4}}{j!}  \| h \|_{\mathcal{\dot{F}}^{2,1}_{\nu}}^{j-1}
\end{equation}
We recall $f_{2}$ in \eqref{infinite.series.fcn.s}.  We have the estimate
\begin{equation} \label{energy14}
\frac{d}{dt} \| h\|_{\mathcal{\dot{F}}^{2,1}_\nu} 
+  (1-b) \| h\|_{\mathcal{\dot{F}}^{6,1}_\nu}
\le  b \| h\|_{\mathcal{\dot{F}}^{2,1}_\nu} +  \| h \|_{\mathcal{\dot{F}}^{6,1}_{\nu}}   f_{2}(\| h \|_{\mathcal{\dot{F}}^{2,1}_{\nu}})
\end{equation}
Recalling \eqref{condition1}, we can choose a small $T>0$ such that
$$
\| h_{0}\|_2 e^{bT} < y_{*}.
$$
Then by choosing $T$ smaller if necessary, on $0\le t \le T$ by continunity  we have
$$
f_{2}(\| h(t)  \|_{\mathcal{\dot{F}}^{2,1}_{\nu}}) \le b_1 < 1,
$$
where $b_1 = b_1(\| h_{0}\|_2, T)$.  
We choose $b>0$ small enough so that $\delta = 1-b-b_1>0$, and then we have
\begin{equation} \label{energy15}
\frac{d}{dt} \| h\|_{\mathcal{\dot{F}}^{2,1}_\nu} + \delta \| h\|_{\mathcal{\dot{F}}^{6,1}_\nu}
\le  b  \| h\|_{\mathcal{\dot{F}}^{2,1}_\nu}. 
\end{equation}
Then we apply the Gr{\"o}nwall inequality to \eqref{energy15} to obtain
$$
  \| h\|_{\mathcal{\dot{F}}^{2,1}_\nu} (t) \le  \| h_{0}\|_2 e^{bT} \le y_{*}.
$$
This completes the proof of the gain of analyticity, and Proposition \ref{prop.local}.

\section{Global existence in $\mathcal{\dot{F}}^{s,p}$}\label{SizeConstant}

In this section we briefly collect our previous estimates and explain the proofs of Theorem \ref{Main.Theorem} and Theorem \ref{Cor.Main.Theorem}.
First if  \eqref{condition1} holds then \eqref{energy-diss2} holds for $s=0$ and for $s=2$.  In particular \eqref{energy3} holds.   This global in time bound combined with Proposition \ref{prop.uniqueness} and Proposition \ref{prop.local} yields directly the proof of Theorem \ref{Main.Theorem}.

  We now explain the proof of Theorem \ref{Cor.Main.Theorem}.  Recall that \eqref{energy-diss2} holds for $s=0$ and for $s=2$.  Then from  \eqref{decay1} the solutions to \eqref{pdeSum} from Theorem \ref{Main.Theorem} have the following slow large time decay rate
\begin{equation}  \label{slowDecay1} 
 \| h(t) \|_2 \le C(\| h_0 \|_2, \| h_0 \|_0) (1+t)^{-1/2}.
\end{equation}
We will use that therefore, after time passes, $ \| h(t) \|_2$ can become smaller.  

Notice that  \eqref{energy4} shows that for some $s > 2$, if $\| h_{0} \|_s$ is bounded then, for the solutions from Theorem \ref{Main.Theorem}, the norm $\| h(t) \|_s$ will remain bounded on any time interval $[0,T]$ with an upper bound that depends on $T$.  Further the decay \eqref{slowDecay1} shows that after a time $T_s>0$ passes then we have that \eqref{condition2} holds at $T_s$ which depends on our choice of $s > 2$.  Then we can conclude the high order s-norm norm decrease in \eqref{energy-diss2} for any $s\ge 0$ assuming only that \eqref{condition1} holds and that  $\| h_{0} \|_s$ is bounded.   We do not need to assume that $\| h_{0} \|_s$ is additionally small.

Further by a similar argument, after a different time $T=T(s,p)>0$ passes then we can show that \eqref{conditionF} holds at time $T$, and then we have the general  $\mathcal{\dot{F}}^{s,p}$ norm decrease in \eqref{energy-dissF} assuming additionally only that $\|   h_{0} \|_{\mathcal{\dot{F}}^{s,p}}$ is bounded.   We do not need to assume that $\|   h_{0} \|_{\mathcal{\dot{F}}^{s,p}}$ is additionally small since \eqref{conditionF} holds at time $T$.

Lastly, the fast large time decay rates \eqref{fastest.decay.rate} follow from Proposition \ref{prop.large.decay} and Proposition \ref{prop.bd.uniform} under the assumptions used in the statement of \eqref{fastest.decay.rate}.

\section{Long time existence and decay in the analytic norms $\dot{\mathcal{F}}^{s,1}_{\nu}$}\label{DecaySectionA}

In this section we will present finally the proof of Theorem \ref{Analytic.Theorem}.  This will show the global in time uniform gain of analyticity with radius of analyticity that grows like $t^{1/4}$ for large $t \gtrsim 1$.  This will also show the uniform large time decay rates of the analytic norms with the optimal linear decay rate as in Remark \ref{remark:optimal}.

\subsection{Gevrey estimates and the radius of analyticity }

We consider $\nu(t)>0$, and now we look at estimates for \eqref{pdeSum} the $\mathcal{\dot{F}}^{s,1}_{\nu}$ space with  $s\ge 0$.  We will show that the radius of analyticity grows like $\nu(t) \approx t^{1/4}$.

We multiply \eqref{Fourier1} by $|\xi|^{s} e^{\nu(t) |\xi| } \overline{\hat h} |{\hat h}|^{-1}(\xi,t)$ to obtain
\begin{align}
   \partial_t \left( |\xi|^{s} e^{\nu(t) |\xi|} |\hat h|(\xi,t) \right)&
    - \nu'(t) |\xi|^{s+1} e^{\nu(t) |\xi| } |\hat h|(\xi,t) 
   + |\xi|^{s+4} e^{\nu(t) |\xi|} |\hat h|(\xi,t) \nonumber \\
    &=  -     \sum_{j=2}^{\infty}   \frac{1}{j!}  |\xi|^{s+2} e^{\nu(t) |\xi|} \overline{\hat h}|{\hat h}|^{-1}(\xi,t)(|\xi|^2 \hat h)^{*j} (\xi,t). 
\label{Fourier771}
\end{align}
To estimate the nonlinear term on the right side we use Young's inequality as 
\begin{multline} 
 \int_{\mathbb{R}^d} |\xi|^{s+2} 
  e^{\nu(t) |\xi|}
 | \overline{\hat h}(\xi) |\hat h|^{-1}(\xi) (|\cdot|^2 \hat h(\cdot))^{*j}(\xi)  |  \,d\xi  
\le 
  \| |\xi|^{s+2} e^{\nu(t) |\xi|} (|\xi|^2 \hat h)^{*j} \|_{L^1_\xi} 
 \\
  \le 
 j^{s + 2} \| (|\xi|^{s+4} e^{\nu(t) |\xi|}\hat h)*(|\xi|^2 e^{\nu(t) |\xi|}\hat h)^{*(j-1)} \|_{L^1_\xi} 
 \\
  \le 
j^{s + 2} \|f\|_{\mathcal{\dot{F}}^{s+4,1}_{\nu}}(t)  \| h \|_{\mathcal{\dot{F}}^{2,1}_{\nu}}^{j-1}.  
\label{convolutionPnu7}
\end{multline} 
Now we use (\ref{convolutionPnu7})  to obtain the following estimate
\begin{equation} \label{energy117}
\frac{d}{dt} \| h\|_{\mathcal{\dot{F}}^{s,1}_\nu} 
\le   
 \nu'(t) \| h\|_{\mathcal{\dot{F}}^{s+1,1}_\nu}
-  
\| h\|_{\mathcal{\dot{F}}^{s+4,1}_\nu}
+
 \| h \|_{\mathcal{\dot{F}}^{s+4,1}_{\nu}} \sum_{j=2}^{\infty}   \frac{j^{s+2}}{j!}  \| h \|_{\mathcal{\dot{F}}^{2,1}_{\nu}}^{j-1}.
\end{equation}
We will use \eqref{energy117} to  simultaneously prove a global bound and large time decay rates.

For now we will focus on the second two terms on the left side of \eqref{energy117}.  
From H{\"o}lder's inequality, and then Young's inequality with $\frac{3}{4}+\frac{1}{4}=1$, then multiply and divide by $\nu(t)^{3/4}$ and we have that
\begin{equation}\notag
\| h\|_{\mathcal{\dot{F}}^{s+1,1}_\nu} 
\le
\| h\|_{\mathcal{\dot{F}}^{s,1}_\nu}^{3/4} 
\| h\|_{\mathcal{\dot{F}}^{s+4,1}_\nu}^{1/4} 
\le
\frac{3}{4}
\nu(t)^{-1}
\| h\|_{\mathcal{\dot{F}}^{s,1}_\nu}
+
\frac{1}{4}
\nu(t)^{3}\| h\|_{\mathcal{\dot{F}}^{s+4,1}_\nu}.
\end{equation}
Further, since $e^x \le 1+x e^x$ for $x\ge 0$ we also have that
\begin{multline}\notag
\| h\|_{\mathcal{\dot{F}}^{s,1}_\nu} = 
\int_{\mathbb{R}^d} |\xi|^{s}e^{\nu(t) |\xi|}|\hat{h}(\xi,t)| d\xi
\\
\le 
\int_{\mathbb{R}^d} |\xi|^{s}|\hat{h}(\xi,t)| d\xi
+
\nu(t)\int_{\mathbb{R}^d} |\xi|^{s+1}e^{\nu(t) |\xi|}|\hat{h}(\xi,t)| d\xi
\\
\le 
\| h\|_{\mathcal{\dot{F}}^{s,1}}
+
\frac{3}{4}
\| h\|_{\mathcal{\dot{F}}^{s,1}_\nu}
+
\frac{\nu(t)^4}{4}
\| h\|_{\mathcal{\dot{F}}^{s+4,1}_\nu}.
\end{multline}
We conclude that 
\begin{equation}\notag
\| h\|_{\mathcal{\dot{F}}^{s,1}_\nu}
\le
4\| h\|_{\mathcal{\dot{F}}^{s,1}}
+
\nu(t)^4
\| h\|_{\mathcal{\dot{F}}^{s+4,1}_\nu}.
\end{equation}

Now we estimate 
\begin{multline}\notag
 \nu'(t) \| h\|_{\mathcal{\dot{F}}^{s+1,1}_\nu}
-  \| h\|_{\mathcal{\dot{F}}^{s+4,1}_\nu}
\\
\le
\frac{3}{4} \nu'(t)\nu(t)^{-1}
\| h\|_{\mathcal{\dot{F}}^{s,1}_\nu}
+
\frac{1}{4} \nu'(t)\nu(t)^{3}
\| h\|_{\mathcal{\dot{F}}^{s+4,1}_\nu}
-  \| h\|_{\mathcal{\dot{F}}^{s+4,1}_\nu}
\\
\le
3 \nu'(t)\nu(t)^{-1}
\| h\|_{\mathcal{\dot{F}}^{s,1}}
+
\frac{3}{4} \nu'(t) \nu(t)^{-1}\nu(t)^4
\| h\|_{\mathcal{\dot{F}}^{s+4,1}_\nu}
\\
+
\frac{1}{4} \nu'(t)\nu(t)^{3}
\| h\|_{\mathcal{\dot{F}}^{s+4,1}_\nu}
-  \| h\|_{\mathcal{\dot{F}}^{s+4,1}_\nu}.
\end{multline}

Now with $b>0$ from Proposition \ref{prop.local},  for $t \ge 0$, we choose 
\begin{equation}
\label{nu.def.global}
\nu(t) = ((bt_0)^4+a t)^{1/4},
\end{equation}
for some $t_0>0$ and $a>0$ to be determined.  Then $\nu'(t) = \frac{a}{4} ((b t_0)^4+a t)^{-3/4}$ and $\nu'(t) = \frac{a}{4} \nu(t)^{-3}$.  Further then 
$$
\nu'(t) \nu(t)^{-1}\nu(t)^4 =  \frac{a}{4},
$$
and 
$
\nu'(t) \nu(t)^{-1} =  \frac{a}{4} \nu(t)^{-4}.
$

And then
$$
 \nu'(t) \| h\|_{\mathcal{\dot{F}}^{s+1,1}_\nu}
-  \| h\|_{\mathcal{\dot{F}}^{s+4,1}_\nu}
\le
\frac{a}{4} \nu(t)^{-4}
\| h\|_{\mathcal{\dot{F}}^{s,1}}
-
\left( 1 - \frac{a}{4} \right)  \| h\|_{\mathcal{\dot{F}}^{s+4,1}_\nu}.
$$
Later we will choose $0<a<4$.  

Now we use 
\begin{equation}\notag
-\| h\|_{\mathcal{\dot{F}}^{s+4,1}_\nu} 
\le 
- \nu(t)^{-4} \| h\|_{\mathcal{\dot{F}}^{s,1}_\nu}
+
4\nu(t)^{-4}\| h\|_{\mathcal{\dot{F}}^{s,1}}.
\end{equation}
We choose $\alpha, \beta \ge 0$ such that $\alpha + \beta =1$.    Then we have
\begin{multline}\notag
 \nu'(t) \| h\|_{\mathcal{\dot{F}}^{s+1,1}_\nu}
-  \| h\|_{\mathcal{\dot{F}}^{s+4,1}_\nu}
\\
\le
\left( \frac{a}{4}+4\alpha \left( 1 - \frac{a}{4} \right) \right) \nu(t)^{-4}
\| h\|_{\mathcal{\dot{F}}^{s,1}}
- 
\alpha \left( 1 - \frac{a}{4} \right) \nu(t)^{-4} \| h\|_{\mathcal{\dot{F}}^{s,1}_\nu}
\\
-
\beta\left( 1 - \frac{a}{4} \right)  \| h\|_{\mathcal{\dot{F}}^{s+4,1}_\nu}.
\end{multline}
These are the main upper bounds that we will use.

Now returning to \eqref{energy117}, we obtain the following differential inequality
\begin{multline}
\label{energy11final}
\frac{d}{dt} \| h\|_{\mathcal{\dot{F}}^{s,1}_\nu} 
+
\delta \nu(t)^{-4} \| h\|_{\mathcal{\dot{F}}^{s,1}_\nu}
\le
\lambda \nu(t)^{-4} \| h\|_{\mathcal{\dot{F}}^{s,1}}
-
\kappa  \| h\|_{\mathcal{\dot{F}}^{s+4,1}_\nu}
\\
+
 \| h \|_{\mathcal{\dot{F}}^{s+4,1}_{\nu}} \sum_{j=2}^{\infty}   \frac{j^{s+2}}{j!}  \| h \|_{\mathcal{\dot{F}}^{2,1}_{\nu}}^{j-1}.
\end{multline}
In the above $\delta>0$ is a small constant, $C_0>0$, and $0<\kappa \eqdef\beta\left( 1 - \frac{a}{4} \right)<1$ can be chosen arbitrarily close to $1$.    Further 
$\lambda \eqdef \left( \frac{a}{4}+4\alpha \left( 1 - \frac{a}{4} \right)\right)>0$ can be chosen to be small.

We use the estimate \eqref{energy11final}, combined with the following procedure to obtain the global decay of the analytic norm with radius \eqref{nu.def.global}.  For now we restrict to the case $s=2$.   We start with the solution from Theorem \ref{Main.Theorem} with initial data satisfying 
$$
\| h_0\|_{\mathcal{\dot{F}}^{2,1}} < y_*.
$$
Further as in Theorem \ref{Cor.Main.Theorem} assume that $h_0 \in \mathcal{\dot{F}}^{2,2}$,  $h_0 \in \mathcal{\dot{F}}^{0,2}$, and $\|h_0\|_{-d,\infty} <\infty$.    Then from \eqref{fastest.decay.rate} we can choose a large time $T_\epsilon>0$ such that for $t\ge 0$ we have
 \begin{equation}\label{decay.rate.epsilon}
\| h(T_\epsilon+t) \|_2 \le \epsilon'  \nu(t)^{-(2+d)},
\end{equation}
where we will choose $\epsilon' >0$ small in a moment.  

From the local existence result in Proposition \ref{prop.local}, we know that equation \eqref{pdeSum} has a gain of analyticity on a time local interval  starting with the initial data described in the previous paragraph.   We take initial data for the gain of analyticity as $\| h(T_\epsilon) \|_2 <  \epsilon$ where $\epsilon = \epsilon(\epsilon',t_0)>0$ is small and $\to 0$ as $\epsilon' \to 0$.   Then for a short time interval $[T_\epsilon, T_\epsilon + 2t_0]$ from Proposition \ref{prop.local} we still have 
$$
\| h(T_\epsilon+t) \|_{\mathcal{\dot{F}}^{2,1}_\mu} <  \epsilon
$$ 
for all $t \in [T_\epsilon, T_\epsilon + 2t_0]$ where in Proposition \ref{prop.local} we use $\mu(t) = bt$.

This is how we choose $t_0>0$ small from \eqref{nu.def.global} to guarantee the above based upon our choice of $\epsilon$.  We use estimate \eqref{energy11final} with $s=2$ starting at time $T_\epsilon + t_0$ with $\nu(t) = ((b t_0)^4+a t)^{1/4}$ as in \eqref{nu.def.global}.  To ease the notation, in the rest of this paragraph we write 
$\tilde{h}(t) = h( T_\epsilon + t_0+t)$ and  
$\tilde{h}_0 = h( T_\epsilon + t_0)$.   Now following the arguments below \eqref{energy2}   using \eqref{energy11final} with $s=2$ we have
\begin{multline}
\label{energy11final.for2}
\frac{d}{dt} \| \tilde{h}\|_{\mathcal{\dot{F}}^{2,1}_\nu} 
+
\delta \nu(t)^{-4} \| \tilde{h}\|_{\mathcal{\dot{F}}^{2,1}_\nu}
\le
\lambda \nu(t)^{-4}  \| \tilde{h}\|_{\mathcal{\dot{F}}^{2,1}}
-
\kappa  \| \tilde{h}\|_{\mathcal{\dot{F}}^{6,1}_\nu}
\\
+
 \| \tilde{h} \|_{\mathcal{\dot{F}}^{6,1}_{\nu}} \sum_{j=2}^{\infty}   \frac{j^{4}}{j!}  \| \tilde{h} \|_{\mathcal{\dot{F}}^{2,1}_{\nu}}^{j-1}
 \\
 \le
\epsilon \lambda \nu(t)^{-(2+d)-4}
-
C_1  \| \tilde{h}\|_{\mathcal{\dot{F}}^{6,1}_\nu}
\le
\epsilon \lambda \nu(t)^{-(2+d)-4}.
\end{multline}
Here we used that  $\| \tilde{h}\|_{\mathcal{\dot{F}}^{2,1}} \lesssim \nu(t)^{-(2+d)}$.
Above we can take $C_1 >0$ since, as above, we can choose $\| \tilde{h} \|_{\mathcal{\dot{F}}^{2,1}_{\nu}}$ to be arbitrarily small.  We multiply by $\nu(t)^{4\delta/a}$ to obtain
\begin{equation}
\notag
\frac{d}{dt} \left(\nu(t)^{4\delta/a}  \| \tilde{h}(t) \|_{\mathcal{\dot{F}}^{2,1}_\nu} \right)
\le
\epsilon \lambda \nu(t)^{4\delta/a-(2+d)-4}.
\end{equation}
Then we integrate to obtain
\begin{equation}
\label{energy.again.2}
\| \tilde{h}(t) \|_{\mathcal{\dot{F}}^{2,1}_\nu} 
\le
\frac{\epsilon \lambda }{\delta - a(2+d)/4} \frac{a}{4} \nu(t)^{-(2+d)}
+
\| \tilde{h}_0 \|_{\mathcal{\dot{F}}^{2,1}_\nu}  \nu(0)^{4\delta/a}\nu(t)^{-4\delta/a}.  
\end{equation}
This concludes the main estimates of this paragraph.  

Now choosing $a>0$ sufficiently small, depending upon the other parameters in \eqref{energy.again.2}, allows us to propagate the assumption that $\| h(T_\epsilon+t_0+t) \|_{\mathcal{\dot{F}}^{2,1}_\nu} <  \epsilon$ for all $t\ge 0$.    Therefore \eqref{energy11final.for2} and \eqref{energy.again.2} hold for all times $t\ge 0$.  We conclude that
\begin{equation}
\label{decay.gevrey}
\| h(T+t) \|_{\mathcal{\dot{F}}^{2,1}_\nu} 
=
\int_{\mathbb{R}^d} |\xi|^{2 }e^{\nu(t) |\xi|}|\hat{h}(\xi,T+t)| d\xi
\lesssim (1+t)^{-(2+d)/4}, 
\end{equation}
which holds uniformly for some fixed $T>0$ and all $t\ge 0$.  Here we recall that $\nu(t)$ is given by \eqref{nu.def.global}.

Note that we can further prove \eqref{decay.gevrey} for any $s\ge 0$ by using the same technique, and obtain the decay rate in \eqref{analytic.decay.rate}.  These estimates now grant Theorem \ref{Analytic.Theorem}.
 \hfill  {\bf Q.E.D.}

\appendix
\section{Numerical simulations}\label{sec:appendixA}

This section contains the plots of numerical simulations conducted by Prof. Tom Witelski.  These are numerical simulations of \eqref{pde} on $0\le x\le 2\pi$ subject to periodic boundary conditions.  They were carried out using a backward Euler
fully-implicit finite difference scheme that was second-order accurate in space. A series of simulations were carried out starting from the initial data, $h_0(x)=A\sin(x)$ with $A>0$, which satisfies $ \| h_0 \|_{2}=A$, where we recall that the norm $ \| \cdot \|_{2}$ is defined as in \eqref{Snorm},  and $\max_x \big| h''_{0}(x) \big| =A$.

\begin{center}

\begin{figure}

\includegraphics[scale=.5]{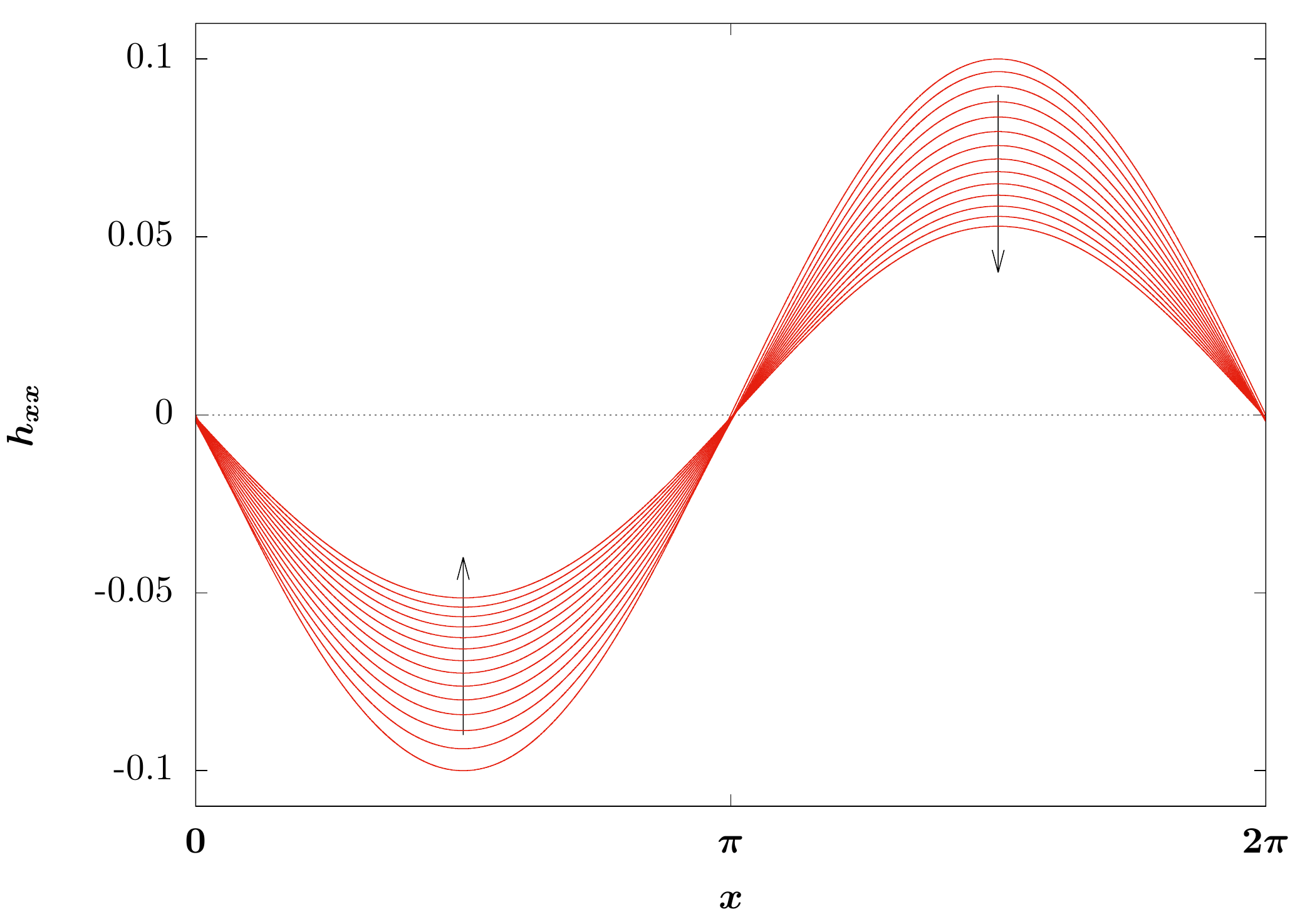}

\includegraphics[scale=.48]{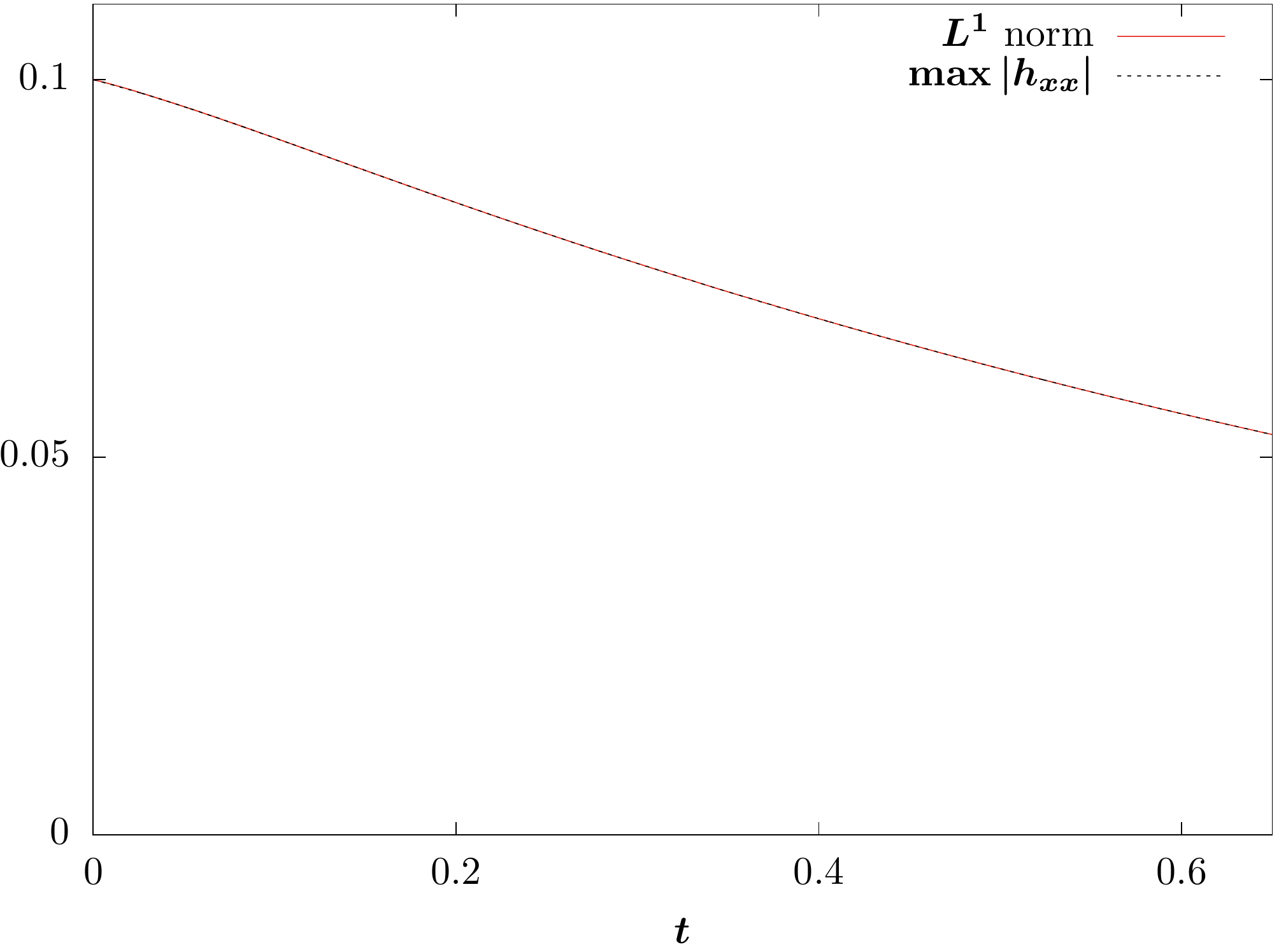}

\caption{Evolution of the solution to the exponential PDE \eqref{pde} with initial data $h_0(x) = 0.1 \times \sin(x)$ in a period domain $[0, 2\pi]$.
The top figure shows some  decreasing height profiles at a sequence of times.  The bottom figure shows the (monotone) time history  of $\| h(\cdot, t)\|_{2}$ 
and $\| h(\cdot, t)\|_{W^{2,\infty}} $.
}

\end{figure}

\end{center}

\eject

\begin{center}

\begin{figure}
\includegraphics[scale=.5]{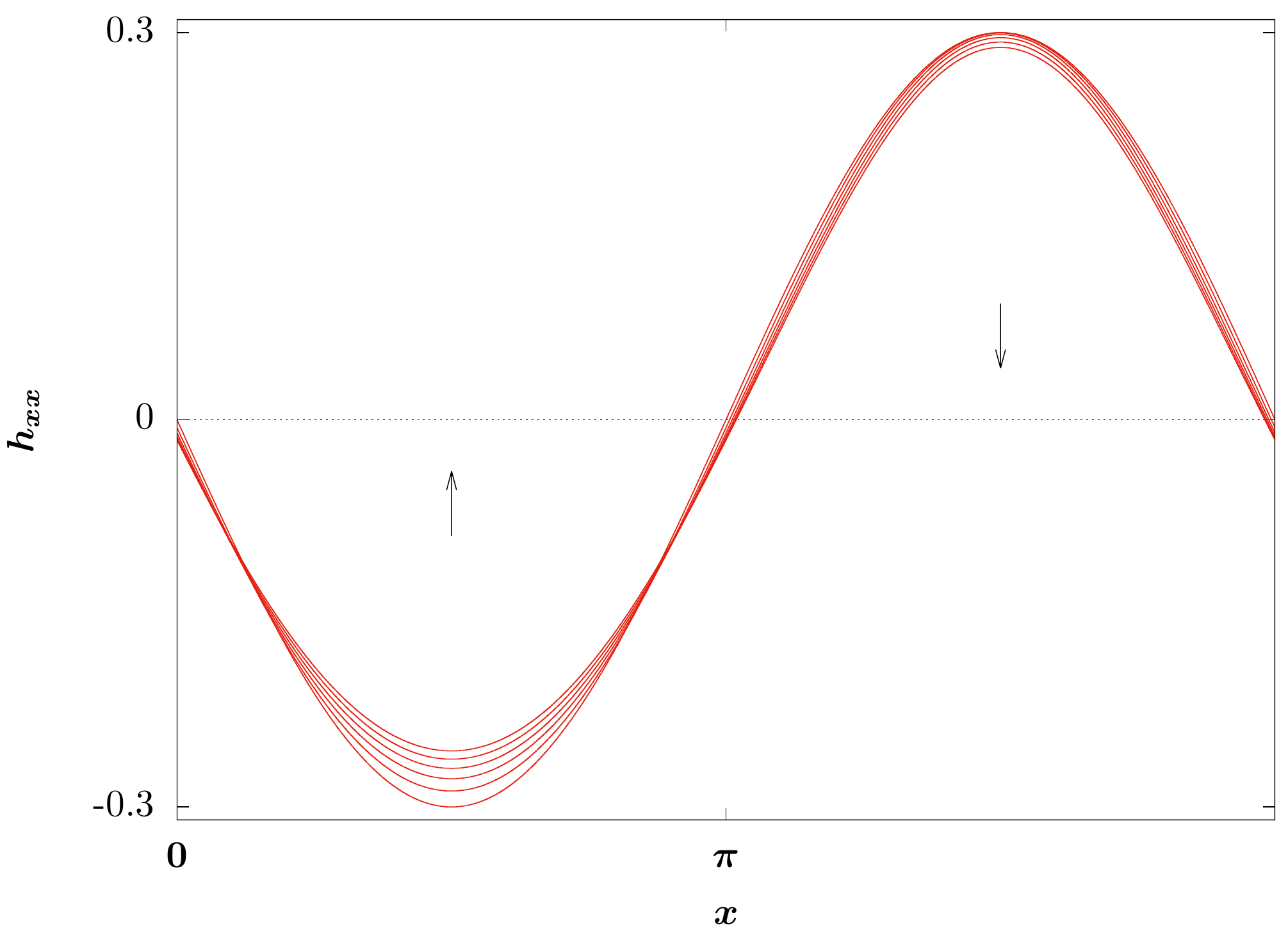}

\includegraphics[scale=.48]{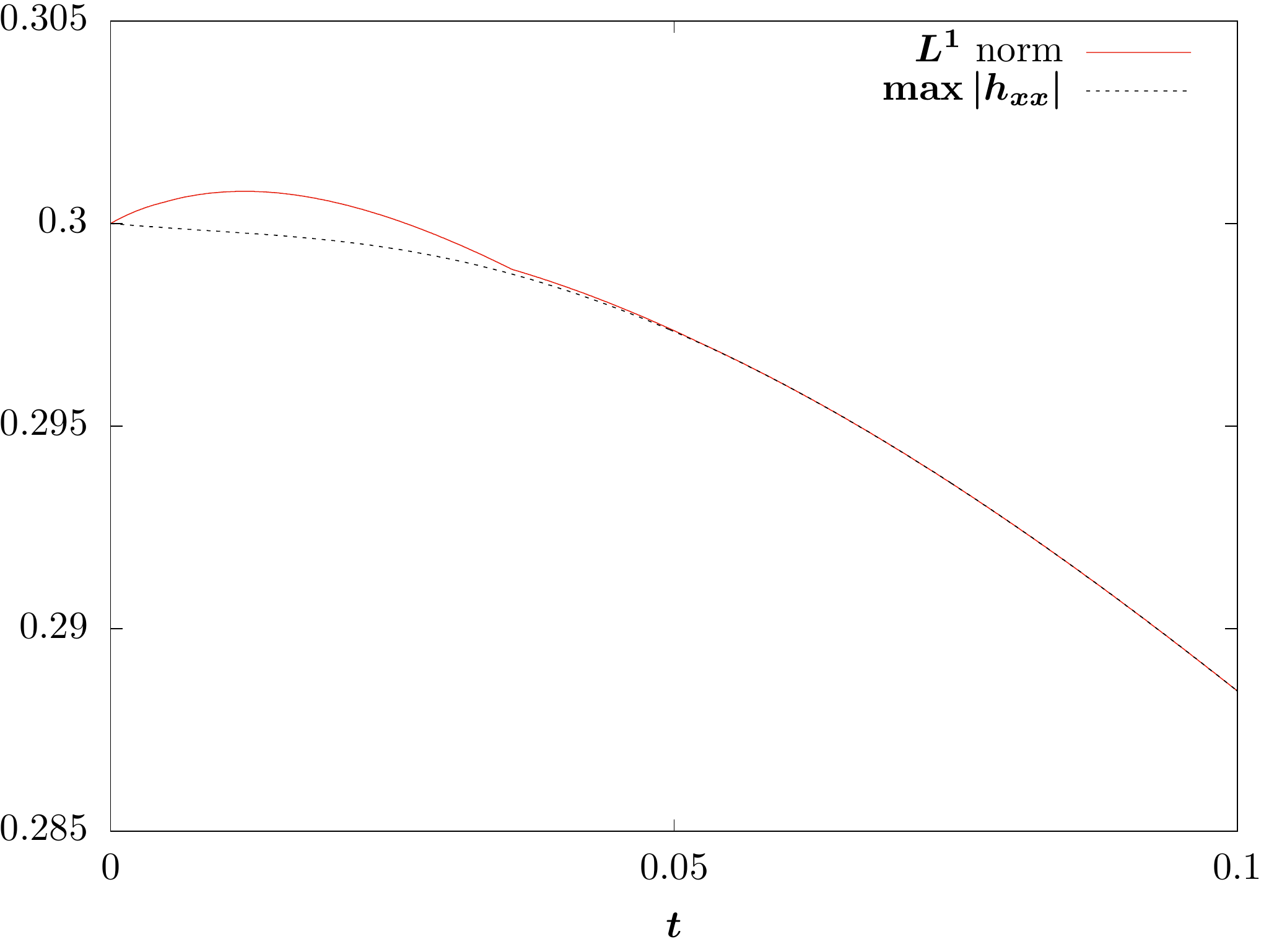}

\caption{Evolution of the solution to exponential PDE \eqref{pde}  with initial data $h_0(x) = 0.3 \times \sin(x)$ in a period domain $[0, 2\pi]$.
The top figure shows some  decreasing hight profiles at a sequence of times.  
The bottom figure shows the (non-monotone) time history  of $\| h(\cdot, t)\|_{2}$ 
and $\| h(\cdot, t)\|_{W^{2,\infty}} $.
}

\end{figure}

\end{center}

\eject

\eject

\begin{center}

\begin{figure}

\includegraphics[scale=.5]{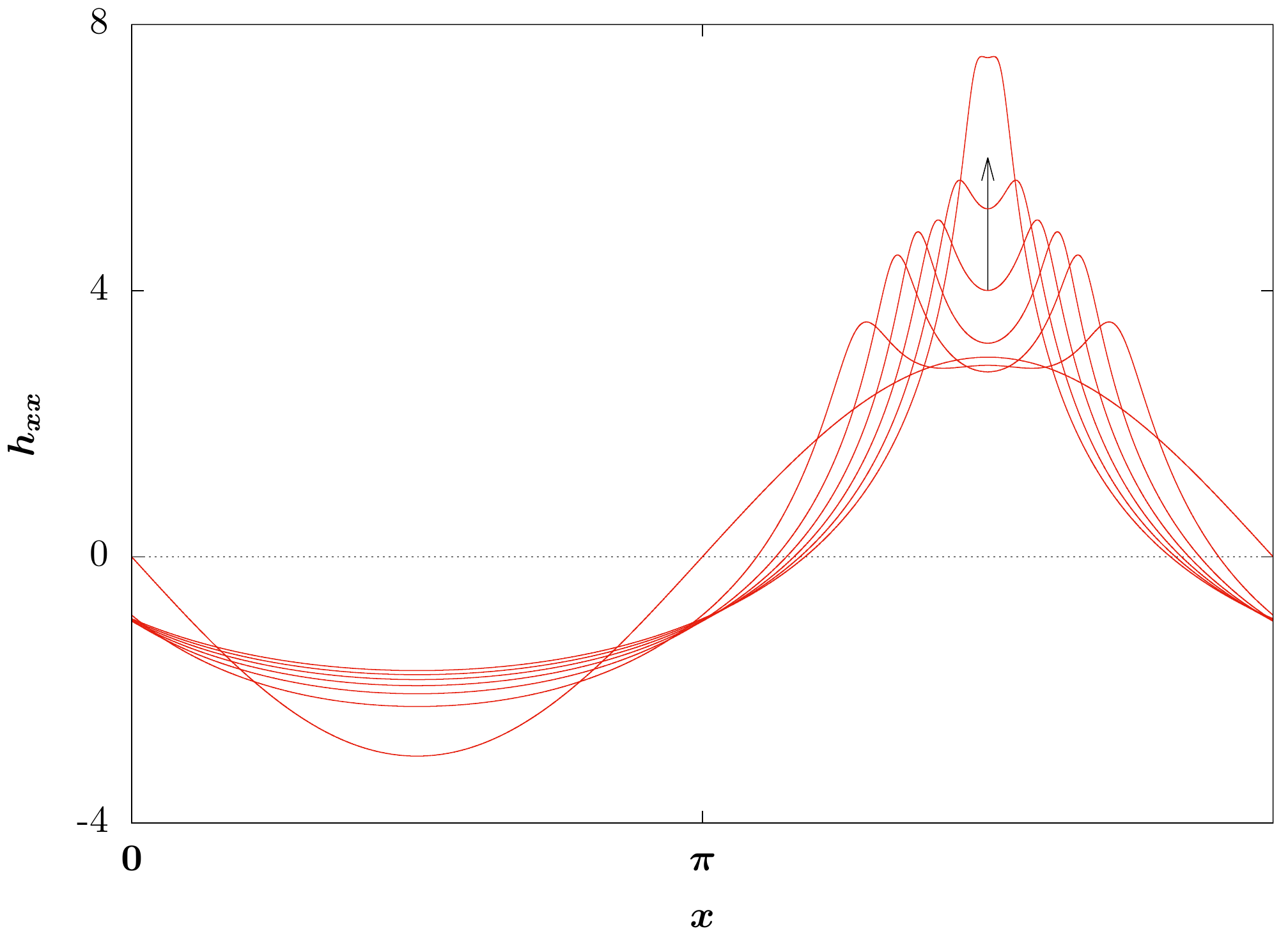}

\includegraphics[scale=.48]{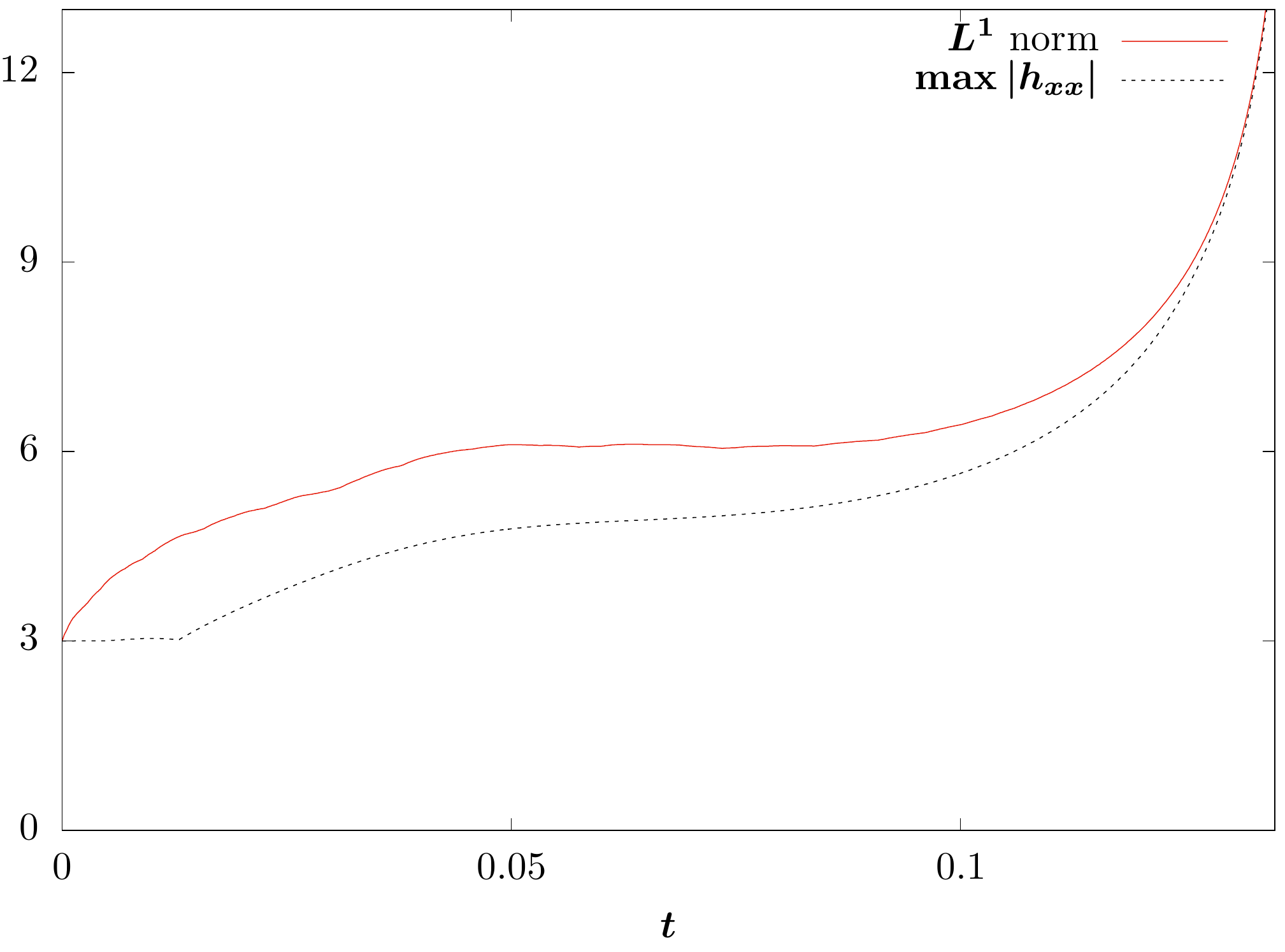}

\caption{Evolution of solution to exponential PDE \eqref{pde} with initial data $h_0(x) = 3 \times \sin(x)$ in a periodic domain $[0, 2\pi]$.
The top figure shows some  blow up hight profiles at a sequence of times.  
The bottom figure shows the blow up time history  of $\| h(\cdot, t)\|_{2}$ 
and $\| h(\cdot, t)\|_{W^{2,\infty}}$.
}

\end{figure}
\end{center}

% these commends below are for a bibtex bibliography
%\bibliographystyle{plain}
%\bibliographystyle{abbrv}
%\bibliographystyle{acm}
%\bibliographystyle{alpha}
%\bibliographystyle{apalike}
%\bibliographystyle{ieeetr}
%\bibliographystyle{siam}
%\bibliographystyle{unsrt}
%\bibliographystyle{plainnat}
%\bibliographystyle{plainurl}
%\bibliographystyle{abbrvnat}
%\bibliographystyle{unsrtnat}
%\bibliographystyle{amsalpha.bst}  % this one is cool with amsrefs
%\bibliographystyle{amsplain.bst}
%% the ones below need to be called with AMSREFS package
%\bibliographystyle{amsrn.bst}  %this one is the default if not style is called
%\bibliographystyle{amsru.bst}
%\bibliographystyle{amsra.bst}
%\bibliographystyle{amsry.bst}
%\bibliographystyle{amsrs.bst}
%\bibliographystyle{amsxport.bst}

%%% the bibliography styles below support arXiv eprint links
%\bibliographystyle{hplain}
%
%\bibliography{exponentialPDE}{}
%
%

\end{document}